\documentclass[12pt,reqno]{amsart}

\RequirePackage{etex}
\usepackage{graphicx}
\usepackage{amsthm,amsmath,amssymb,bbm,fullpage}
\usepackage{url,hyperref}
\usepackage{fontawesome}

\usepackage{epstopdf}
\usepackage[table]{xcolor}
\usepackage[T1]{fontenc}
\usepackage{marvosym}
\usepackage{appendix}
\usepackage[utf8]{inputenc}
\usepackage{braket}
\usepackage{multirow}
\usepackage{rotating}
\usepackage{caption}
\usepackage{subcaption}
\usepackage{mathrsfs}
\usepackage{comment}
\usepackage{bm}
\usepackage[all]{xy}

\newtheorem{theorem}{Theorem}[section]
\newtheorem{lemma}[theorem]{Lemma}
\newtheorem{proposition}[theorem]{Proposition}
\newtheorem{corollary}[theorem]{Corollary}

\newtheorem{conjecture}[theorem]{Conjecture}
\theoremstyle{definition}
  \newtheorem{remark}[theorem]{Remark}
  
  \newtheorem{example}[theorem]{Example}
\numberwithin{equation}{section}

\newcommand{\RR}{\mathbb{R}}
\newcommand{\ZZ}{\mathbb{Z}}

\newcommand\valpha{\bm{\alpha}}
\newcommand\vphi{\bm{\varphi}}

\newcommand\va{\bm{a}}
\newcommand\vc{\bm{c}}
\newcommand\ve{\bm{e}}

\newcommand\vt{\bm{t}}
\newcommand\vs{\bm{s}}
\newcommand\vp{\bm{p}}
\newcommand\vu{\bm{u}}

\newcommand\F{\mathcal{F}}

\DeclareMathOperator{\vol}{vol}
\DeclareMathOperator{\linspan}{span}

\newcommand{\PS}{\textup{PS}}
\newcommand{\Zig}{\textup{Zig}}
\newcommand{\Car}{\textup{Car}}
\newcommand{\PF}{\textup{PF}}

\newcommand{\M}{\mathcal{M}}%{\textup{M}}
\newcommand{\U}{\mathcal{U}}%{\textup{U}}
\newcommand{\GD}{\mathcal{GD}}%{\textup{GD}}
\newcommand{\calH}{\mathcal{H}}%{\textup{H}}

\definecolor{applegreen}{rgb}{0.55,0.71,0.0}

\definecolor{darkcandyapplered}{rgb}{0.64, 0.0, 0.0}

\usepackage{float}
\usepackage{tikz}
\usepackage{tkz-graph}
\tikzstyle{vertex}=[circle, draw, inner sep=0pt, minimum size=4pt]
\newcommand{\vertex}{\node[vertex]}
\tikzstyle{vtx}=[circle, draw, inner sep=0pt, minimum size=12pt]

\definecolor{darkgreen}{cmyk}{.9,0,.9,.2}

\newcommand\bbinom[2]%
{\mathchoice{\left(\kern-0.48em{\binom{#1}{#2}}\kern-0.48em\right)}
            {\bigl(\kern-0.3em{\binom{#1}{#2}}\kern-0.3em\bigr)}
            {\bigl(\kern-0.3em{\binom{#1}{#2}}\kern-0.3em\bigr)}
            {\bigl(\kern-0.3em{\binom{#1}{#2}}\kern-0.3em\bigr)}}

\begin{document}
\title{A combinatorial model for computing \\ volumes of flow polytopes}

\author[Benedetti]{Carolina Benedetti}
\address[C.\ Benedetti]{Departamento de Matem\'aticas\\Universidad de los Andes\\Bogot\'a\\Colombia} 
\email{c.benedetti@uniandes.edu.co}
\urladdr{\url{https://sites.google.com/site/carobenedettimath/home}}

\author[Gonz\'alez D'Le\'on]{Rafael S. Gonz\'alez D'Le\'on }
\address[R.\ S.\ Gonz\'alez D'Le\'on]{Escuela de Ciencias Exactas e Ingenier\'ia\\Universidad Sergio Arboleda\\Bogot\'a\\Colombia} 
\email{rafael.gonzalezl@usa.edu.co}
\urladdr{\url{http://dleon.combinatoria.co}}

\author[Hanusa]{Christopher R.\ H.\ Hanusa}
\address[C.\ R.\ H.\ Hanusa]{Department of Mathematics \\ Queens College (CUNY) \\ 65-30 Kissena Blvd. \\ Flushing, NY 11367\\ United States}
\email{\href{mailto:chanusa@qc.cuny.edu}{\texttt{chanusa@qc.cuny.edu}}}
\urladdr{\url{http://qc.edu/~chanusa/}}

\author[Harris]{Pamela E. Harris}
\address[P.\ E.\ Harris]{Mathematics and Statistics, Williams College, Bascom House, Rm 106C
Bascom House, 33 Stetson Court, Williamstown, MA 01267, United States} 
\email{peh2@williams.edu}
\urladdr{\url{https://math.williams.edu/profile/peh2/}}

\author[Khare]{Apoorva Khare}
\address[A.\ Khare]{Department of Mathematics, Indian Institute of Science; Analysis and Probability Research Group; Bangalore 560012, India}
\email{khare@iisc.ac.in}
\urladdr{\url{http://www.math.iisc.ac.in/~khare/}}

\author[Morales]{Alejandro H. Morales}
\address[A.\ H.\ Morales]{Department of Mathematics and Statistics, University of Massachusetts, Amherst, MA, 01003, United States} 
\email{ahmorales@math.umass.edu}
\urladdr{\url{http://people.math.umass.edu/~ahmorales/}}

\author[Yip]{Martha Yip}
\address[M.\ Yip]{Department of Mathematics, University of Kentucky, 715 Patterson Office Tower, Lexington, KY 40506-0027, United States}
\email{martha.yip@uky.edu}
\urladdr{\url{http://www.ms.uky.edu/~myip/}}

%\comment{
%05A15 Exact enumeration problems, generating functions
%05A19 Combinatorial identities, bijective combinatorics
%52B05 Combinatorial properties [of polytopes]
%52A38 Length, area, volume [in discrete geometry]
%05C20 Directed graphs (digraphs), tournaments
%05C21 Flows in graphs
%52B11 $n$-dimensional polytopes
%}


% Activate to display a given date or no date
% \date{\today}

\begin{abstract}
We introduce new families of combinatorial objects whose enumeration computes volumes of flow polytopes.  These objects provide an interpretation, based on parking functions, of Baldoni and Vergne's generalization of a volume formula originally due to Lidskii.  
We recover known flow polytope volume formulas and prove new volume formulas for flow polytopes.  A highlight of our model is an elegant formula for the flow polytope of a graph we call the caracol graph. 

As by-products of our work, we uncover a new triangle of numbers that interpolates between Catalan numbers and the number of parking functions, we prove the log-concavity of rows of this triangle along with other sequences derived from volume computations, and we introduce a new Ehrhart-like polynomial for flow polytope volume and conjecture product formulas for the polytopes we consider.
\end{abstract}

\maketitle
\begin{center}
  \emph{Dedicated to the memory of Griff L.\ Bilbro.}
\end{center}

\section{Introduction}
Flow polytopes are a family of polytopes with remarkable enumerative and geometric properties.  They are related to several areas of mathematics including toric geometry~\cite{Hille03}, representation theory (Verma modules~\cite{Humphreys}), special functions (the Selberg integral~\cite{ZLF}), and algebraic combinatorics (diagonal harmonics~\cite{MMR} and Schubert polynomials~\cite{MSt}). Their combinatorial and geometric study started with work of Baldoni and Vergne~\cite{BV08} and unpublished work of Postnikov and Stanley. 

For an acyclic directed graph $G$ on $n+1$ vertices and $m$ edges accompanied by an integer vector $\va=(a_1,\dots,a_n)\in\ZZ^{n}$, the flow polytope $\F_G(\va)\subset\RR^m$ encodes the set of flows on $G$ with net flow on its vertices given by~$\va':=(a_1,\dots,a_n,-\sum_{i=1}^n a_i)$. The number of such integer-valued flows is the number of integer lattice points of $\F_G(\va)$, which is known in the literature as a {\em Kostant partition function} and is denoted by $K_G(\va')$. Equivalently, this quantity counts the number of ways that $\va'$ can be written as a sum of the vectors $\ve_i-\ve_j$ corresponding to the edges $(i,j)$ of $G$. Also of interest is the normalized volume of $\F_G(\va)$ (hereafter called simply the volume of $\F_G(\va)$).  For an $n$-dimensional lattice polytope $P$ with Euclidean volume $V(P)$, its normalized volume is defined to be $n!V(P)$; for instance, the normalized volume of the unit cube in $\RR^n$ is $n!$.

For certain graphs $G$ and vectors $\va$, the volume and number of lattice points of $\F_G(\va)$ have nice combinatorial formulas. We highlight a few examples involving various special graphs and the vectors $\va=(1,0,\ldots,0)$ and $\va=(1,1,\ldots,1)$. 
See~\cite{CKM,MeszarosProd,MMS,MSW} for other examples.  

\begin{enumerate}
\item[(i)] When $G$ is the complete graph $K_{n+1}$ and $\va=(1,0,\ldots,0)$, $\F_G(\va)$ is called the Chan--Robbins--Yuen polytope~\cite{CRY}.  
As proved by Zeilberger~\cite{Z}, its volume is the product of consecutive Catalan numbers: 
\[
\vol \F_{K_{n+1}}(1,0,\ldots,0) \,=\, \prod_{i=1}^{n-2} C_i,
\]
where $C_k :=\frac{1}{k+1}\binom{2k}{k}$.
\item[(ii)] When $G$ is the complete graph $K_{n+1}$ and $\va=(1,1,\ldots,1)$, $\F_G(\va)$ is called the  Tesler polytope.  As proved by M\'esz\'aros, Morales and Rhoades~\cite{MMR}, its volume also features a product of Catalan numbers:
\[\vol \F_{K_{n+1}}(1,1,\ldots,1) = \frac{\binom{n}{2}!}{\prod_{i=1}^{n-2} (2i + 1)^{n - i - 1}} \cdot \prod_{i=1}^{n-1} C_i.
\]
The only known proofs of (i) and (ii) use a variant of the {\em Morris constant term identity}. The original Morris identity is equivalent to the famous Selberg integral. 
\item[(iii)] When $G$ is the zigzag graph $\Zig_{n+1}$  consisting of a path $1\to 2 \to \cdots \to n+1$ and additional edges $(i,i+2)$ for $1\leq i\leq n-1$ (see Figure~\ref{figure:zigzag}), and $\va=(1,0,\ldots,0)$, the polytope $\F_{\Zig_{n+1}}(\va)$ has volume
\begin{equation}\label{equation:Eulernumber}
\vol\F_{\Zig_{n+1}}(1,0,\ldots,0) = E_{n-1},
\end{equation}
which is half the number of alternating permutations on $n-1$ letters (see  Stanley's survey~\cite{Stanley_SurveyAP}). 

\begin{figure}[htb]
    \centering
\begin{tikzpicture}
	\vertex[fill,label=below:\tiny{$1$}](a0) at (0,0) {};
	\vertex[fill,label=below:\tiny{$2$}](a1) at (1,0) {};
	\vertex[fill,label=below:\tiny{$3$}](a2) at (2,0) {};
		%\node at (2.5,0) {$\cdots$};
	\vertex[fill,label=below:\tiny{$4$}](a3) at (3,0) {};
	\vertex[fill,label=below:\tiny{$5$}](a4) at (4,0) {};
		\node at (4.5,0) {$\cdots$};
	\vertex[fill,label=above :\tiny{$n-1$}](a10) at (5,0) {};
	\vertex[fill,label=above:\tiny{$n$}](a11) at (6,0) {};
	\vertex[fill,label=above:\tiny{$n+1$}](a12) at (7,0) {};
	
	\draw[->] (a0)--(.92,0);		
	\draw[->] (1,0)--(1.92,0);
	\draw[->] (2,0)--(2.92,0);
	\draw[->] (3,0)--(3.92,0);
	\draw (4,0)--(4.2,0);
	\draw[->] (4.7,0)--(4.92,0);
	\draw[->] (5,0)--(5.92,0);
	\draw[->] (6,0)--(6.92,0);

	\draw[->]    (a0) to[out=50,in=130] (a2);
	\draw[->]    (a2) to[out=50,in=130] (a4);
	\draw[->]    (a4) to[out=60,in=200] (4.5,.45);	
	\draw[->]    (5.5,.45) to[out=0,in=110] (a11);	
	
	\draw[->]    (a1) to[out=-50,in=-130] (a3);	
	\draw[->]    (a3) to[out=-50,in=-170] (4.5,-.5);	
	\draw[->]    (a10) to[out=-50,in=-130] (a12);	
	
\end{tikzpicture}
    \caption{Zigzag graph $\Zig_{n+1}$.}
    \label{figure:zigzag}
\end{figure}
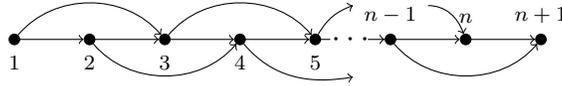

\item[(iv)] When $G$ is the Pitman--Stanley graph denoted by $\PS_{n+1}$ and consisting of a path $1\to 2 \to \cdots \to n+1$ and additional edges $(1,n+1), (2,n+1),\ldots,(n-1,n+1)$ (see Figure~\ref{figure:PSGraph}),
and $\va = (1,1,\ldots,1)$, the polytope $\F_{\PS_{n+1}}(\va)$ is affinely equivalent to the Pitman--Stanley polytope~\cite{PitmanStanley2002}. The number of lattice points in $\F_{\PS_{n+1}}(\va)$ is
$K_{\PS_{n+1}}(1,1,\ldots,1,-n) = C_n
$ and its volume is
\begin{equation}
\vol\F_{\PS_{n+1}}(1,1,\ldots,1) = n^{n-2}. \label{equation:PitmanStanleyvolume}
\end{equation}
Pitman and Stanley gave two combinatorial proofs of these results.

\item[(v)] 
Consider the graph on $n+1$ vertices created by adding to the graph $\PS_n$ (with indices incremented by one) an extra source vertex $1$ and the edges $(1,2)$, $(1,3)$, \ldots, $(1,n)$.  We call this graph the \emph{caracol graph} and denote it by $\Car_{n+1}$. (Caracol is the Spanish word for snail; the name comes from its embedding---see Figure \ref{figure:CaracolGraph}.) 

\begin{figure}[htb]
\centering
\begin{minipage}{.54\textwidth}
  \centering
\begin{tikzpicture}[scale=0.95]
    \draw[->,white]    (a1) to[out=45,in=130] (a12);
    \vertex[fill,label=below:\tiny{$1$}](a1) at (1,0) {};
	\vertex[fill,label=below:\tiny{$2$}](a2) at (2,0) {};
		%\node at (2.5,0) {$\cdots$};
	\vertex[fill,label=below:\tiny{$3$}](a3) at (3,0) {};
	\vertex[fill,label=below:\tiny{$4$}](a4) at (4,0) {};
		\node at (4.5,0) {$\cdots$};
	\vertex[fill,label=above :\tiny{$n-1$}](a10) at (5,0) {};
	\vertex[fill,label=above:\tiny{$n$}](a11) at (6,0) {};
	\vertex[fill,label=above:\tiny{$n+1$}](a12) at (7,0) {};

	\draw[->] (1,0)--(1.92,0);
	\draw[->] (2,0)--(2.92,0);
	\draw[->] (3,0)--(3.92,0);
	\draw (4,0)--(4.2,0);
	\draw[->] (4.7,0)--(4.92,0);
	\draw[->] (5,0)--(5.92,0);
%	\draw[->] (6,0)--(6.92,0);
%	\draw[->] (7,0)--(7.92,0);
%	\draw[->] (8,0)--(8.92,0);
%	\draw (9,0)--(9.2,0);
%	\draw[->] (9.7,0)--(9.92,0);
%	\draw[->] (10,0)--(10.92,0);
%	\draw[->]    (a0) to[out=50,in=130] (a7);
	\draw[->]    (a11) to (a12);
	\draw[->]    (a10) to[out=-50,in=205] (a12);
%	\draw[->]    (a9) to[out=-50,in=210] (a12);
%	\draw[->]    (a8) to[out=-50,in=215] (a12);
%	\draw[->]    (a7) to[out=-50,in=220] (a12);
%	\draw[->]    (a6) to[out=-50,in=225] (a12);
%	\draw[->]    (a5) to[out=-50,in=230] (a12);
	\draw[->]    (a4) to[out=-50,in=235] (a12);
	\draw[->]    (a3) to[out=-50,in=240] (a12);
	\draw[->]    (a2) to[out=-50,in=240] (a12);
	\draw[->]    (a1) to[out=-50,in=240] (a12);

\end{tikzpicture}
  \caption{Pitman--Stanley~graph~$\PS_{n+1}$.\hspace{-.5in}}\label{figure:PSGraph}
\end{minipage}\!\!\!\!\!%
\begin{minipage}{.49\textwidth}
  \centering
\begin{tikzpicture}[scale=0.95]
	\vertex[fill,label=below:\tiny{$1$}](a0) at (0,0) {};
	\vertex[fill,label=below:\tiny{$2$}](a1) at (1,0) {};
	\vertex[fill,label=below:\tiny{$3$}](a2) at (2,0) {};
		%\node at (2.5,0) {$\cdots$};
	\vertex[fill,label=below:\tiny{$4$}](a3) at (3,0) {};
	\vertex[fill,label=below:\tiny{$5$}](a4) at (4,0) {};
		\node at (4.5,0) {$\cdots$};
	\vertex[fill,label=above :\tiny{$n-1$}](a10) at (5,0) {};
	\vertex[fill,label=above:\tiny{$n$}](a11) at (6,0) {};
	\vertex[fill,label=above:\tiny{$n+1$}](a12) at (7,0) {};

	\draw[->] (1,0)--(1.92,0);
	\draw[->] (2,0)--(2.92,0);
	\draw[->] (3,0)--(3.92,0);
	\draw (4,0)--(4.2,0);
	\draw[->] (4.7,0)--(4.92,0);
	\draw[->] (5,0)--(5.92,0);
%	\draw[->] (6,0)--(6.92,0);
%	\draw[->] (7,0)--(7.92,0);
%	\draw[->] (8,0)--(8.92,0);
%	\draw (9,0)--(9.2,0);
%	\draw[->] (9.7,0)--(9.92,0);
%	\draw[->] (10,0)--(10.92,0);
	\draw[->]    (a0) -- (a1);
	\draw[->]    (a0) to[out=25,in=130] (a2);
	\draw[->]    (a0) to[out=30,in=130] (a3);
	\draw[->]    (a0) to[out=35,in=130] (a4);
	\draw[->]    (a0) to[out=40,in=130] (a10);
	\draw[->]    (a0) to[out=45,in=130] (a11);
%	\draw[->]    (a0) to[out=50,in=130] (a7);
	\draw[->]    (a11) to (a12);
	\draw[->]    (a10) to[out=-50,in=205] (a12);
%	\draw[->]    (a9) to[out=-50,in=210] (a12);
%	\draw[->]    (a8) to[out=-50,in=215] (a12);
%	\draw[->]    (a7) to[out=-50,in=220] (a12);
%	\draw[->]    (a6) to[out=-50,in=225] (a12);
%	\draw[->]    (a5) to[out=-50,in=230] (a12);
	\draw[->]    (a4) to[out=-50,in=235] (a12);
	\draw[->]    (a3) to[out=-50,in=240] (a12);
	\draw[->]    (a2) to[out=-50,in=240] (a12);
	\draw[->]    (a1) to[out=-50,in=240] (a12);
\end{tikzpicture}
  \caption{Caracol~graph~$\Car_{n+1}$.}\label{figure:CaracolGraph}
\end{minipage}
\end{figure}

\noindent
When $G$ is the graph $\Car_{n+1}$ and $\va=(1,0,\ldots,0)$, M\'esz\'aros, Morales, and Striker~\cite{MMS} used an observation of Postnikov to prove that the polytope $\F_{\Car_{n+1}}(\va)$ is affinely equivalent to the order polytope of the poset $[2]\times[n-2]$.  Then by standard properties of order polytopes \cite{Stanley_Order_Poly}, the volume of the polytope is
\begin{equation}
\vol \F_{\Car_{n+1}}(1,0,\ldots,0) = C_{n-2} \label{eq:volCar100}
\end{equation}
by counting the number of linear extensions of that poset.
\end{enumerate}

In this article, we introduce new combinatorial structures, called \emph{gravity diagrams} and \emph{unified diagrams}, which are based on classical combinatorial objects known as parking functions and whose enumeration provide combinatorial interpretations for terms in the Lidskii formulas of Baldoni and Vergne~\cite{BV08}.  In particular, Theorem~\ref{theorem:gravitydiagrams} proves that gravity diagrams provide a combinatorial interpretation for $K_G(\va')$, and Theorem~\ref{theorem:unifieddiagrams} proves that unified diagrams give a combinatorial interpretation of $\vol \F_G(\va)$.  Figures~\ref{figure:zigzag_gravity} and \ref{figure:example_caracol_5_11111} respectively exemplify gravity diagrams and unified diagrams.

These objects permit us to give new combinatorial proofs for lattice point counts and volumes of flow polytopes such as Equations~\eqref{equation:Eulernumber}, \eqref{equation:PitmanStanleyvolume}, and \eqref{eq:volCar100}.  They also allow us to compute in Theorem~\ref{thm.main} the following new elegant formula for the volume of the flow polytope for $\Car_{n+1}$ with net flow vector $\va=(1,1,\ldots,1)$,
\[
\vol \F_{\Car_{n+1}}(1,1,\ldots,1) = C_{n-2}\cdot n^{n-2},
\]
and to obtain formulas for $\vol \F_{G}(\va)$ for new net flow vectors, including $\va=(1,0,\ldots,0,1,0,\ldots,0)$ and $\va=(a,b,b,\ldots,b)$.  For example, in Proposition~\ref{prop:volCar1100} and Theorem~\ref{thm:Car-abb} we prove the two equations
\[\vol\F_{\Car_{n+1}}(1,1,0,\hdots,0) 
= C_{n-2}\cdot n\cdot 2^{n-3}
\]\[\vol\F_{\Car_{n+1}}(a,b,\ldots,b) 
= C_{n-2}\cdot  a^{n-2}\big(a+(n-1)b\big)^{n-2}
\]
combinatorially, without appealing to constant term identities.
Moreover, we are able to show via the Aleksandrov--Fenchel inequalities that the sequences associated to a graph $G$ and net flow vector $\va$ defined by our refined unified diagrams are log-concave.
This includes the Entringer numbers 
(which are entries of the {\em Euler--Bernoulli Triangle}~\cite[\href{https://oeis.org/A008282}{A008282}]{OEIS}) and entries in a fascinating new triangle of numbers that we call the {\em parking triangle} (see Table~\ref{table:parkingtriangle}), which arises from our study of the volume of the caracol polytope. The entries of the parking triangle interpolate between the Catalan numbers and the number of parking functions; we prove in Theorem~\ref{thm:triangleclosedformulas} that the entries  satisfy 
\begin{equation*}
T(n,k) = (n+1)^{k-1} \binom{2n-k}{n},
\end{equation*}
and we provide a parking function interpretation for these numbers involving distinct cars~$\textup{\small \faCar\,}$ and identical motorcycles $\textup{\small \faMotorcycle\,}$.

Lastly, we introduce a new Ehrhart-like polynomial $E_G(x)$ for the volume of flow polytopes with nonnegative coefficients which warrants future research.  For instance, we propose in Conjecture~\ref{conj:preCaracol} that
\[E_{\Car_{n+1}}(x) = \frac{1}{xn+n-3}\binom{xn+2n-5}{n-1} \binom{x+n-3}{n-2}\]
is the volume of the flow polytope for the caracol graph in which one additional vertex $0$ is added and $x$ independent edges are added from $0$ to each other vertex.

The organization of this article is as follows. Section~\ref{sec:background} provides background on flow polytopes, Kostant partition functions, and the Lidskii formulas.  
Section ~\ref{sec:VectorInterps} introduces gravity diagrams for general graphs and net flow vectors, and shows that they provide a combinatorial interpretation for the Kostant partition function $K_G(\va')$.
Section~\ref{section:LidskiiInterp} builds on gravity diagrams and parking functions to create unified diagrams, and proves that they provide a combinatorial interpretation for $\vol \F_G(\va)$. 
Section~\ref{section:volumecaracol} focuses on proving the formula for $\vol \F_{\Car_{n+1}}(1,1,\ldots,1)$ via the enumeration of unified diagrams which includes the discussion of the parking triangle.  
Section~\ref{sec:othervols} applies our techniques to calculate the volume of flow polytopes for new vectors $\va$, while Section~\ref{sec:logconcave} proves the log concavity of related sequences of Kostant partition functions and numbers of refined unified diagrams.  
Section~\ref{sec:polynomial} introduces our new polynomial approach to calculate volume of flow polytopes.  We propose several conjectures in Sections~\ref{sec:othervols} and \ref{sec:polynomial}.

%%%%%%%%%%%%%%%%%%%%%%%%%%%%%%%%%%%%%%%%%%%%%%%%%%%%%%%%%%%%%
%%%%%%%%%%%%%%%%%%%%%%%%%%%%%%%%%%%%%%%%%%%%%%%%%%%%%%%%%%%%%
\section{Background}\label{sec:background}

%%%%%%%%%%%%%%%%%%%%%%%%%%%%%%
\subsection{Flow polytopes}\label{sec:flow}

Let $G$ be a connected directed graph on the vertex set $V(G)=[n+1]:=\{1,2,\ldots,n+1\}$ and directed edge set $E(G)$, where each edge is directed from a smaller numbered vertex to a larger numbered vertex.  Let $m=\lvert E(G)\rvert$ denote the number of edges of $G$. Given such a graph $G$ and a vector $\va=(a_1,a_2,\ldots,a_n)$ in $\mathbb{Z}^{n}$, an $\va$-flow on $G$ is a tuple $(b_{ij})_{(i,j) \in E(G)}$ of nonnegative real numbers such that for $j=1,\ldots,n$,
\begin{align}\label{equation:nodeflowsums}
\sum_{(j,k) \in E(G)} b_{jk} - \sum_{(i,j) \in E(G)} b_{ij} = a_j.
\end{align}
We view an $\va$-flow on $G$ as an assignment of flow $b_{ij}$ to each edge $(i,j)$ such that the net flow at each vertex $j\in [n]$ is $a_j$ and such that the net flow at vertex $n+1$ is $-\sum_{j=1}^n a_j$. Let $\F_G(\va)$ denote the set of $\va$-flows of $G$. We view $\F_G(\va)$ as a polytope in $\mathbb{R}^{m}$ and we call it the {\em flow polytope} of $G$ with net flow $\va$. In this article, both $\va$ and $\va'=(a_1,\dots,a_n,-\sum_{i=1}^n a_i)$ will be referred to as the {\em net flow vector} depending on the context, since they refer to the same information.

%%%%%%%%%%%%%%%%%%%%%%%%%%%%%%
\subsection{Kostant partition functions}  We can also view an $\va$-flow on $G$ as a linear combination of the roots \[\valpha_i := \ve_i-\ve_{i+1} \textup{ for $1\leq i\leq n$}\] 
in the type $A$ root system.  (The vector $\ve_i\in \ZZ^n$ is the elementary basis vector with a one in entry $i$ and zeroes elsewhere.) Associate to every directed edge $(i,j)\in E(G)$ the vector
\[
(i,j) \mapsto \ve_i-\ve_j = \valpha_i + \cdots + \valpha_{j-1},
\]
and let $\Phi^+_G$ denote the set of such vectors.  Note that $\Phi_G^+$ is a subset of the set $\Phi^+$ of positive roots of the Coxeter system of type $A_n$.
In this setting, Equation \eqref{equation:nodeflowsums} is equivalent to writing $\va'$ as a linear combination of the vectors $\valpha_i + \cdots + \valpha_{j-1}$:
\begin{align}\label{equation:vectorpartition}
\va' = \sum_{(i,j) \in E(G)} b_{ij}[\valpha_i + \cdots + \valpha_{j-1}].
\end{align}
When the $\va$-flow $(b_{ij})$ is integral, this is an instance of a {\em vector partition} of $\va'$. The number of integral $\va$-flows on $G$ is called the {\em Kostant partition function} of $G$ evaluated at $\va'$ and we denote it by $K_G(\va')$.

Vector partition functions play an important role in many areas of mathematics.  The Kostant partition function of the complete graph computes weight multiplicities for important classes of highest weight representations of Lie groups and Lie algebras of type $A$ (see Humphreys~\cite[Section 24]{Humphreys}). These are the parabolic Verma modules over the Lie algebra $\mathfrak{sl}_{n+1}$ of traceless $(n+1) \times (n+1)$ matrices, which include Verma modules over $\mathfrak{sl}_{n+1}$ and integrable polynomial representations of the Lie group $GL_{n+1}(\mathbb{C})$, the latter indirectly through the Weyl Character Formula.

%%%%%%%%%%%%%%%%%%%%%%%%%%%%%%
\subsection{The Lidskii formula}
Baldoni and Vergne~\cite{BV08} proved a remarkable formula for calculating the volume of a flow polytope using residue techniques. They called it the {\em Lidskii formula}. This formula was also proved by Postnikov and Stanley (unpublished, see \cite{Stanley_slides}) and by M\'esz\'aros and Morales~\cite{MM2} using polytope subdivisions. 
To state this formula we first define a few concepts.

Given a directed graph $G$, the  {\em shifted out-degree vector } $\vt=(t_1,\ldots, t_n)$ is the vector whose $i$-th entry is one less than the out-degree of vertex $i$.  The sum of the entries of $\vt$ is denoted by $\lvert\vt\rvert$ and always equals $m-n$.  Furthermore we use the notation $G|_n$ to denote the subgraph of $G$ that is the restriction of the graph $G$ to the first $n$ vertices.

A {\em weak composition} of $n$ is a sequence of nonnegative integers whose sum is $n$.  Given weak compositions $\vs=(s_1,\dots,s_n)$ and $\vt=(t_1,\dots,t_n)$ of $n$, we say {\em $\vs$ dominates $\vt$} and write $\vs \rhd \vt$ if $\sum_{i=1}^k s_i \geq \sum_{i=1}^k t_i$ for every $k\geq1$. We also use the standard notation for the multiexponent
\[\va^{\vs}:=a_1^{s_1}a_2^{s_2}\cdots a_n^{s_n}\] 
and for the multinomial coefficient
\[\binom{k}{\vs}:=\frac{k!}{s_1!s_2!\cdots s_n!}.\]

\begin{theorem}[Lidskii volume formula, {\cite[Theorem 38]{BV08}}, {\cite[Theorem 1.1]{MM2}}] \label{thm.Lidskii} 
Let $G$ be a directed graph with $n+1$ vertices, $m$ edges, and shifted out-degree vector $\vt=(t_1,\ldots, t_n)$ and let $\va = (a_1,\ldots,a_n)$ be a nonnegative integer vector.  Then
\begin{align} \label{eq:lidskiivol}
\vol\mathcal{F}_G(\va)
= \sum_{\vs\rhd \vt} \binom{m-n}{\vs}\cdot \va^{\vs}\cdot 
K_{G|_n}(\vs-\vt),
\end{align} 
where the sum is over weak compositions $\vs= (s_1,\ldots, s_n)$ of $m-n$  
that dominate $\vt$.  
\end{theorem}

\begin{remark}
In  \cite[Theorem 38]{BV08} and \cite[Theorem~1.1]{MM2} a formula is presented, similar to the one in Equation~\eqref{eq:lidskiivol}, for the Kostant partition function $K_G(\va')$.
\end{remark}

For flow polytopes where $\va=(1,0,\dots,0)$, the sum in Equation~\eqref{eq:lidskiivol} has only one nonzero term, namely when $\vs=(m-n,0,\dots,0)$. In this case, the equation has the following well-known simplification due to Postnikov and Stanley, up to applying a symmetry of the Kostant partition function (see, for example, \cite[Propositition 2.2]{MM2}).

\begin{corollary}[Postnikov and Stanley~\cite{Stanley_slides}]
\label{cor:volume}
Let $G$ be a directed graph with $n+1$ vertices,  $m$ edges, and shifted out-degree vector $\vt=(t_1,\ldots, t_n)$.  Then

\begin{equation}
\label{eq:lidskiivol100}
\begin{aligned} 
\vol\mathcal{F}_G(1,0,\dots,0)
&= K_{G|_n}(m-n-t_1,-t_2,\dots,-t_n)\\
&=K_{G|_n}\bigg(\sum_{i=1}^{n-1}\bigg[m-n-\sum_{j=1}^{i}t_j\bigg]\valpha_i\bigg).
\end{aligned} 
\end{equation}
\end{corollary}

In other words, Corollary~\ref{cor:volume} says that we can calculate the volume of the flow polytope $\mathcal{F}_G(1,0,\dots,0)$ by counting the number of lattice points in a related flow polytope, which we exhibit in the following three examples.  The appearance of the vector $(n-2,-1,\dots,-1,0)$ in each example is coincidental.

\begin{example}
\label{ex:PitmanStanleyVector}
The Pitman--Stanley graph $G=\PS_{n+1}$ has out-degree vector
$(2,2,\ldots,2,1)$ and $m=2n-1$ edges.  The shifted out-degree vector of $G$ is 
$\vt=(1,1,\dots,1,0),$
so we conclude from Corollary~\ref{cor:volume} that 
\[
\begin{aligned}
\vol\mathcal{F}_{\PS_{n+1}}(1,0,\ldots,0) &= K_{(\PS_{n+1})|_n}(n-2,-1,\ldots,-1,0)\\
&= K_{(\PS_{n+1})|_n}\big((n-2)\valpha_1+(n-3)\valpha_2+\cdots+\valpha_{n-2}\big).
\end{aligned}
\]
\end{example}

\begin{example}
\label{ex:CaracolVector}
The caracol graph $G=\Car_{n+1}$ has out-degree vector
$(n-1,2,\ldots,2,1)$ and $m=3n-4$ edges.  The shifted out-degree vector of $G$ is 
$\vt=(n-2,1,1,\dots,1,0).$  We have
\[
\begin{aligned}
\vol\mathcal{F}_{\Car_{n+1}}(1,0,\dots,0) &= K_{(\Car_{n+1})|_n}(n-2,-1,\dots,-1,0)\\
&= K_{(\Car_{n+1})|_n}\big((n-2)\valpha_1+(n-3)\valpha_2+\cdots+\valpha_{n-2}\big).
\end{aligned}
\]
\end{example}

\begin{example}
\label{ex:ZigzagVector}
The zigzag graph $G=\Zig_{n+1}$ has out-degree vector
$(2,2,\ldots,2,1)$ and $m=2n-1$ edges. The shifted out-degree vector of $G$ is $\vt=(1,1,1,\dots,1,0).$ 
  The restriction $(\Zig_{n+1})|_n$ is simply the graph $\Zig_n$, so we have
\[
\begin{aligned}
\vol\mathcal{F}_{\Zig_{n+1}}(1,0,\dots,0) &= K_{\Zig_n}(n-2,-1,\dots,-1,0)\\
&= K_{\Zig_n}\big((n-2)\valpha_1+(n-3)\valpha_2+\cdots+\valpha_{n-2}\big).
\end{aligned}
\]
\end{example}

%%%%%%%%%%%%%%%%%%%%%%%%%%%%%%%%%%%%%%%%%%%%%%%%%%%%%%%%%%%%%
%%%%%%%%%%%%%%%%%%%%%%%%%%%%%%%%%%%%%%%%%%%%%%%%%%%%%%%%%%%%%
\section{A combinatorial interpretation for vector partitions}
\label{sec:VectorInterps}

In this section, we define gravity diagrams as a combinatorial interpretation of the right hand side of Equation~\eqref{eq:lidskiivol100}. 

%%%%%%%%%%%%%%%%%%%%%%%%%%%%%%
\subsection{Gravity diagrams}\label{section:gravitydiagrams}
Let $\vc=(c_1,\ldots,c_n)$ be a nonnegative integer vector.
We define the concept of a line-dot diagram to give a visualization of an integral 
$\vc$-flow with net flow $\vc'=(c_1,\ldots,c_n,-\sum_{i=1}^n c_i)$.  We use the equivalent notion of a vector partition of $\vc'$ that has parts of the form $[\valpha_i+\cdots+\valpha_{j-1}]$, where $1\leq i<j\leq n$, with multiplicity when there are multiple edges.  With intuition coming from the fact that 
\begin{align}
\vc'&=c_1\ve_1+c_2\ve_2+\cdots+c_n\ve_n-(c_1+\cdots+c_n)\ve_{n+1}\\
&=(c_1)\valpha_1+(c_1+c_2)\valpha_2+\cdots+(c_1+\cdots+c_n)\valpha_n,
\end{align}
we create a two-dimensional array of dots with $n$ columns and $c_1+\cdots+c_i$ dots in column~$i$.  These dots represent the copies of $\valpha_i$ that must be included in the vector partition.  By convention, we align the dots with their bottom-most dots in the same row, but in later parts of this article (like Section \ref{section:unifieddiagrams}), it will be convenient to align the dots differently.
The parts in a vector partition are of the form $\valpha_i+\cdots+\valpha_{j-1}$, so for each part of this type, we represent it by drawing a ``line'' (in fact, a connected path of line segments) connecting one dot from each of columns~$i$ through $j-1$.  A part consisting of a single vector $\valpha_i$ is represented by a line segment of length~0 (a dot) in column~$i$.  We call such a diagram a {\em line-dot diagram}.  
Multiple line-dot diagrams may correspond to the same partition of $\vc'$.  Figure~\ref{fig:doublecount2} shows two distinct line-dot diagrams for the vector $3\valpha_1+4\valpha_2+2\valpha_3+\valpha_4$, each representing the vector partition
$[\valpha_1]+[\valpha_2]+[\valpha_1+\valpha_2]+[\valpha_1+\valpha_2+\valpha_3]+[\valpha_2+\valpha_3+\valpha_4].$

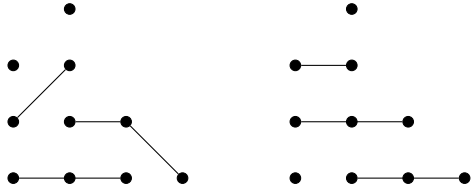
\begin{figure}[htbp]
\centering
\begin{tikzpicture}[scale=0.75]
\tikzstyle{every node}=[circle, fill, draw, inner sep=0pt, minimum size=4pt]
\node (a11) at (-1,1) {};
\node (a12) at (-2,1) {};
\node (a22) at (-2,2) {};
\node (a13) at (-3,1) {};
\node (a23) at (-3,2) {};
\node (a33) at (-3,3) {};
\node (a14) at (-4,1) {};
\node (a24) at (-4,2) {};
\node (a34) at (-4,3) {};
\node (a44) at (-3,4) {};
\draw (a11)--(a13);
\draw (a22)--(a24);
\draw (a33)--(a34);
\node (b11) at (-6,1) {};
\node (b12) at (-7,1) {};
\node (b22) at (-7,2) {};
\node (b13) at (-8,1) {};
\node (b23) at (-8,2) {};
\node (b33) at (-8,3) {};
\node (b14) at (-9,1) {};
\node (b24) at (-9,2) {};
\node (b34) at (-9,3) {};
\node (b44) at (-8,4) {};
\draw (b12)--(b14);
\draw (b11)--(b22);
\draw (b22)--(b23);
\draw (b24)--(b33);
\end{tikzpicture}
\caption{Distinct line-dot diagrams corresponding to the same partition of the vector $3\valpha_1+4\valpha_2+2\valpha_3+\valpha_4$.}
\label{fig:doublecount2}
\end{figure}

Two line-dot diagrams are said to be equivalent if they represent the same vector partition of $\vc'$, and we let $\GD_G(\vc')$ denote the set of such equivalence classes. An
equivalence class is called a {\em gravity diagram} (whose name will become intuitive in Section~\ref{sec:gravitycaracol}) and the choice of a class representative will depend on the flow polytope.  By construction we have the following result.

\begin{theorem}\label{theorem:gravitydiagrams}
For any graph $G$ on $n+1$ vertices and for any net flow vector $\vc'\in\mathbb{Z}^{n}$,
\[K_{G}(\vc')=\big\lvert\GD_G(\vc')\big\rvert.\]
\end{theorem}

We highlight the use of gravity diagrams in the computation of flow polytope volumes in the next three subsections.

%%%%%%%%%%%%%%%%%%%%%%%%%%%%%%
\subsection{Gravity diagrams for the Pitman--Stanley graph}

The Pitman--Stanley graph restricted to the first $n$ vertices contains only edges from $i$ to $i+1$ for $1\leq i\leq n-1$, so $K_{(\PS_{n+1})|n}(\vc')$ counts vector partitions of ${\bf c'}=c_1 \ve_1 + c_2 \ve_2 + \cdots + c_{n-1}\ve_{n-1}$ involving parts of type
\begin{center}
\begin{tabular}{ll}
$[\valpha_i]$ & for all $1\leq i\leq n-1$.\\
\end{tabular}
\end{center}
The line-dot diagrams for these vector partitions consist of triangular arrays using dots only, with no lines allowed.

\begin{proposition} \label{prop:GravityPS}
For the Pitman--Stanley graph $\PS_{n+1}$ on $n+1$ vertices and for any net flow vector $\vc'\in\mathbb{Z}^{n}$,
\[
K_{(\PS_{n+1})\mid_n}(\vc') = \begin{cases}
1, & \text{ if } \vc' \in \linspan_{\mathbb{N}}(\valpha_1,\ldots,\valpha_{n-1}),\\
0, & \text{ otherwise}.
\end{cases}
\]
\end{proposition}

\begin{proof}
This follows since there is a unique gravity diagram in $\big\lvert\GD_{(\PS_{n+1})\mid_n}(\vc')\big\rvert$ (that only contains dots) only when $\vc'$ is in $\linspan_{\mathbb{N}}(\valpha_1,\ldots,\valpha_{n-1})$. 
\end{proof}

\subsection{Gravity diagrams for the caracol graph} \label{sec:gravitycaracol}

We will use gravity diagrams to calculate $\vol\mathcal{F}_G(1,0,\dots,0)$ for the caracol graph $G=\Car_{n+1}$; this was previously shown in~\cite{MMS} to be equal to the Catalan number $C_{n-2}$ by showing that $\mathcal{F}_G(1,0,\dots,0)$  is affinely equivalent to an order polytope. Example~\ref{ex:CaracolVector} shows that we must calculate $K_{G|_n}(\vc')$ for
\[\vc'=(n-2)\valpha_1+(n-3)\valpha_2+\cdots+\valpha_{n-2}.\]
Figure~\ref{figure:restricte_caracol_graph} shows $(\Car_{n+1})|_n$, the restriction of the caracol graph to the first $n$ vertices, and highlights that we are counting vector partitions of $\vc'$ with parts of the following two types:
\begin{center}
\begin{tabular}{ll}
$[\valpha_i]$ &  for all $1\leq i\leq n-1$, and   \\
$[\valpha_1+\valpha_2+\cdots+\valpha_i]$ &  for all $2\leq i\leq n-1$.
\end{tabular}
\end{center}
\begin{figure}[htb]
    \centering
    \begin{tikzpicture}[scale=1.4]
\tikzstyle{every node}+=[circle,inner sep=0pt, minimum size=4pt, scale=1.2]
	\node[fill,label=below:\tiny{$1$}](a0) at (0,0) {};
	\node[fill,label=below:\tiny{$2$}](a1) at (1,0) {};
	\node[fill,label=below:\tiny{$3$}](a2) at (2,0) {};
	\node at (2.5,0) {$\cdots$};
    \node[fill,label=below:\tiny{$n-1$}](a11) at (3,0) {};
	\node[fill,label=below:\tiny{$n$}](a11) at (4,0) {};
	\draw[->] (0,0)--(0.92,0) node [midway, above]  {\tiny$\valpha_1$};
	\draw[->] (1,0)--(1.92,0) node [midway, above]  {\tiny$\valpha_2$};
	\draw (1,0)--(2.2,0) ;
	\draw[->] (2.7,0)--(2.92,0);
	\draw[->] (3,0)--(3.92,0) node [midway,above]  {\tiny$\valpha_{n-1}$};
\end{tikzpicture}
    \qquad
    \begin{tikzpicture}[scale=1.4]
\tikzstyle{every node}+=[circle,inner sep=0pt, minimum size=4pt, scale=1.2]
	\node[fill,label=below:\tiny{$1$}](a0) at (0,0) {};
	\node[fill,label=below:\tiny{$2$}](a1) at (1,0) {};
	\node[fill,label=below:\tiny{$3$}](a2) at (2,0) {};
	\node[fill,label=below:\tiny{$4$}](a3) at (3,0) {};
	\node[fill,label=below:\tiny{$5$}](a4) at (4,0) {};
	\node at (4.5,0) {$\cdots$};
	\node[fill,label={[label distance=-0.1cm]below :\tiny{$n-1$}}](a10) at (5,0) {};
	\node[fill,label=below:\tiny{$n$}](a11) at (6,0) {};
	\draw[->] (0,0)--(0.92,0) node [midway, below]  {\tiny$\valpha_1$};
	\draw[->] (1,0)--(1.92,0) node [midway, below]  {\tiny$\valpha_2$};
	\draw[->] (2,0)--(2.92,0) node [midway, below]  {\tiny$\valpha_3$};
	\draw[->] (3,0)--(3.92,0) node [midway, below]  {\tiny$\valpha_4$};
	\draw (4,0)--(4.2,0) ;
	\draw[->] (4.7,0)--(4.92,0);
	\draw[->] (5,0)--(5.92,0)node [midway,inner sep=-6,below ]  {\tiny$\valpha_{n-1}$};
	\draw[->]    (a0) to[out=70,in=-220] (a2)  ;
	\draw[->]    (a0) to[out=70,in=-220] (a3);
	\draw[->]    (a0) to[out=70,in=-220] (a4);
	\draw[->]    (a0) to[out=70,in=-220] (a10);
	\draw[->]    (a0) to[out=70,in=-220] (a11);
\node[rotate=-25] at (1.6,0.4) {\tiny$\valpha_1+\valpha_2$};
\node[rotate=-28] at (2.5,0.5) {\tiny$\valpha_1+\valpha_2+\valpha_3$};
\node[rotate=-28] at (3.3,0.6) {\tiny$\valpha_1+\valpha_2+\valpha_3+\valpha_4$};
\node[rotate=-28] at (4.2,0.7) {\tiny$\valpha_1+\valpha_2+\cdots+\valpha_{n-2}$};
\node[rotate=-28] at (5,0.8) {\tiny$\valpha_1+\valpha_2+\cdots+\valpha_{n-1}$};

\end{tikzpicture}
    \caption{The graphs $(\PS_{n+1})|_n$ and $(\Car_{n+1})|_{n}$ with each edge labeled by its corresponding positive root.}
    \label{figure:restricte_caracol_graph}
\end{figure}
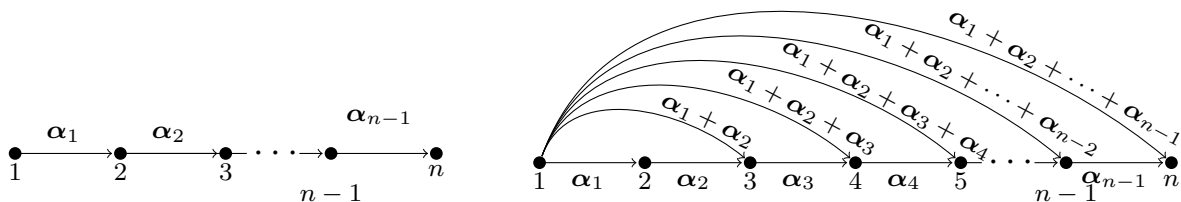

The line-dot diagrams that correspond to these vector partitions consist of a triangular array of dots arranged into $n-2$ columns with $n-1-i$ dots in the $i$-th column for $1\leq i \leq n-2$.
The lines allowed in this case are the ones that extend from column $1$ to column $i$, including lines of length zero.
In this case, a canonical class representative of $\GD_{(\Car_{n+1})|_{n}}(\vc')$ exclusively uses horizontal lines starting in the first column arranged in terms of length from shortest to longest going from top to bottom.  This looks like a stack of line segments with the longest or ``heaviest'' lines at the bottom, from which the name gravity diagram originated.  See Figure~\ref{figure:caracol_gravity} for an example of the gravity diagrams in this collection.

\begin{figure}[htb]
    \centering
    \begin{tikzpicture}[scale=0.8]
\begin{scope}[xshift=0]
%\node[scale=1] at (2,0){$\mathbf{v}=(0,0)$};
\tikzstyle{every node}+=[fill,circle, inner sep=0, minimum size=4pt]
\node(a11) at (1,1) {};
\node(a12) at (1,2) {};
\node(a13) at (1,3) {};
\node(a21) at (2,1) {};
\node(a22) at (2,2) {};
\node(a31) at (3,1) {};

\end{scope}

\begin{scope}[xshift=110]
%\node[scale=1] at (2,0){$\mathbf{v}=(1,0)$};

\tikzstyle{every node}+=[fill,circle, inner sep=0, minimum size=4pt]
\node(a11) at (1,1) {};
\node(a12) at (1,2) {};
\node(a13) at (1,3) {};
\node(a21) at (2,1) {};
\node(a22) at (2,2) {};
\node(a31) at (3,1) {};
\draw (a11)--(a21);
\end{scope}

\begin{scope}[xshift=220]
%\node[scale=1] at (2,0){$\mathbf{v}=(1,1)$};

\tikzstyle{every node}+=[fill,circle, inner sep=0, minimum size=4pt]
\node(a11) at (1,1) {};
\node(a12) at (1,2) {};
\node(a13) at (1,3) {};
\node(a21) at (2,1) {};
\node(a22) at (2,2) {};
\node(a31) at (3,1) {};
\draw (a11)--(a31);
\end{scope}

\begin{scope}[xshift=330]
%\node[scale=1] at (2,0){$\mathbf{v}=(0,1)$};

\tikzstyle{every node}+=[fill,circle, inner sep=0, minimum size=4pt]
\node(a11) at (1,1) {};
\node(a12) at (1,2) {};
\node(a13) at (1,3) {};
\node(a21) at (2,1) {};
\node(a22) at (2,2) {};
\node(a31) at (3,1) {};
\draw (a11)--(a21);
\draw (a12)--(a22);
\end{scope}

\begin{scope}[xshift=440]
%\node[scale=1] at (2,0){$\mathbf{v}=(0,2)$};

\tikzstyle{every node}+=[fill,circle, inner sep=0, minimum size=4pt]
\node(a11) at (1,1) {};
\node(a12) at (1,2) {};
\node(a13) at (1,3) {};
\node(a21) at (2,1) {};
\node(a22) at (2,2) {};
\node(a31) at (3,1) {};
\draw (a11)--(a31);
\draw (a12)--(a22);
\end{scope}
\end{tikzpicture}
    \caption{The set of gravity diagrams in $\GD_{(\Car_6)|_5}(3,-1,-1,-1,0)$.}
    \label{figure:caracol_gravity}
\end{figure}
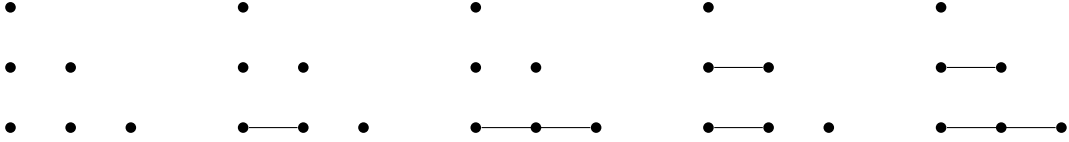

We can now give a new proof for Equation~\eqref{eq:volCar100} by way of a visual bijection between such gravity diagrams and $N$-$E$ lattice paths from $(0,0)$ to $(n-2,n-2)$ lying below the main diagonal.  An example of this bijection is shown in Figure~\ref{figure:DyckBijection}.

\begin{proposition}\label{prop:Dyck1} The volume of the flow polytope $\vol \F_{\Car_{n+1}}(1,0,\dots,0)$ is $C_{n-2}$.
\end{proposition}
\begin{proof} Given a gravity diagram $D$,
reflect it about the vertical axis and embed it into the integer lattice $\mathbb{Z}\times \mathbb{Z}$
with the bottom row of $D$ occupying the lattice points $(1,1)$ through $(n-2,1)$ and leftmost column of $D$ occupying the lattice points $(n-2,1)$ through $(n-2,n-2)$. 
We now construct the lattice path $P$ corresponding to $D$.
Start at $(n-2,n-2)$ and follow a vertical step down to row $n-3$. In each row, if the vertical step meets a horizontal line of the gravity diagram, follow the line to its left-most endpoint, at which position take a vertical step downward to the row below it. Continue in this fashion until you arrive at the left-most endpoint in the first row of $D$, and finish the path by taking one step downward and continuing the path to $(0,0)$.  The reader can check that this construction provides a bijection between gravity diagrams and lattice paths from $(0,0)$ to $(n-2,n-2)$ lying below the diagonal, which are classical Dyck paths and known to be counted by $C_{n-2}$.
\end{proof}

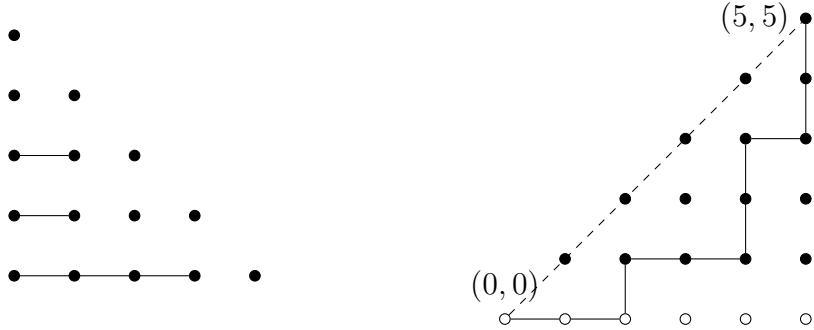
\begin{figure}
\centering
\begin{tikzpicture}[scale=0.8]
\node(a00) at (0,0) {};
\vertex[fill,label=below:](a11) at (-1,1) {};
\vertex[fill,label=below:](a12) at (-2,1) {};
\vertex[fill,label=below:](a13) at (-3,1) {};
\vertex[fill,label=below:](a14) at (-4,1) {};
\vertex[fill,label=below:](a15) at (-5,1) {};
\vertex[fill,label=below:](a22) at (-2,2) {};
\vertex[fill,label=below:](a23) at (-3,2) {};
\vertex[fill,label=below:](a24) at (-4,2) {};
\vertex[fill,label=below:](a25) at (-5,2) {};
\vertex[fill,label=below:](a33) at (-3,3) {};
\vertex[fill,label=below:](a34) at (-4,3) {};
\vertex[fill,label=below:](a35) at (-5,3) {};
\vertex[fill,label=below:](a44) at (-4,4) {};
\vertex[fill,label=below:](a45) at (-5,4) {};
\vertex[fill,label=below:](a55) at (-5,5) {};
\draw (a12)--(a15);
\draw (a24)--(a25);
\draw (a34)--(a35);
\end{tikzpicture}
\qquad\qquad
\begin{tikzpicture}[scale=0.8]
\vertex[label=above:{$(0,0)$}](a00) at (0,0) {};
\vertex[label=below:](a01) at (1,0) {};
\vertex[label=below:](a02) at (2,0) {};
\vertex[label=below:](a03) at (3,0) {};
\vertex[label=below:](a04) at (4,0) {};
\vertex[label=below:](a05) at (5,0) {};
\vertex[fill,label=below:](a11) at (1,1) {};
\vertex[fill,label=below:](a12) at (2,1) {};
\vertex[fill,label=below:](a13) at (3,1) {};
\vertex[fill,label=below:](a14) at (4,1) {};
\vertex[fill,label=below:](a15) at (5,1) {};
\vertex[fill,label=below:](a22) at (2,2) {};
\vertex[fill,label=below:](a23) at (3,2) {};
\vertex[fill,label=below:](a24) at (4,2) {};
\vertex[fill,label=below:](a25) at (5,2) {};
\vertex[fill,label=below:](a33) at (3,3) {};
\vertex[fill,label=below:](a34) at (4,3) {};
\vertex[fill,label=below:](a35) at (5,3) {};
\vertex[fill,label=below:](a44) at (4,4) {};
\vertex[fill,label=below:](a45) at (5,4) {};
\vertex[fill,label=left:{$(5,5)$}](a55) at (5,5) {};
\draw (a00)--(a01)--(a02)--(a12)--(a14)--(a34)--(a35)--(a55);
\draw[dashed] (a00)--(a55);
\end{tikzpicture}
\caption{A gravity diagram of $(\Car_8)|_7$ and its corresponding lattice path described in the bijection of Proposition~\ref{prop:Dyck1}.}
\label{figure:DyckBijection}
\end{figure}

%%%%%%%%%%%%%%%%%%%%%%%%%%%%%%
\subsection{\bf Gravity diagrams for the zigzag graph}\label{sec:gravityzigzag}

The zigzag graph $\Zig_{n+1}$, when restricted to its first $n$ vertices is $\Zig_n$, contains edges of the form $(i,i+1)$ and $(i,i+2)$.  
Recalling Example~\ref{ex:ZigzagVector}, we once again count vector partitions of $\vc'=(n-2)\valpha_1+(n-3)\valpha_2+\cdots+\valpha_{n-2},$
but this time we count those involving parts of types
\begin{center}
\begin{tabular}{ll}
$[\valpha_i]$ &  for all $1\leq i\leq n-1$, and   \\
$[\valpha_i+\valpha_{i+1}]$ &  for all $1\leq i\leq n-2$.
\end{tabular}
\end{center}

The line-dot diagrams for these vector partitions consist of the same triangular array of dots as for the caracol graph, but with lines that only connect dots in two consecutive columns. A canonical gravity diagram in $\GD_{\Zig_n}(\vc')$ follows the convention that the lines are placed from right to left, and each line occupies the lowest available dots in their respective columns. See Figure~\ref{figure:zigzag_gravity} for an example of the gravity diagrams in this collection.  

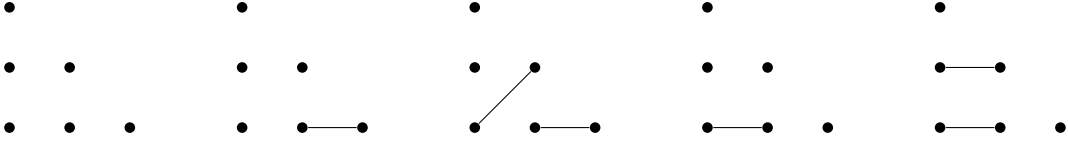
\begin{figure}[htb]
    \centering
    \begin{tikzpicture}[scale=0.8]
\begin{scope}[xshift=0]
%\node[scale=1] at (2,0){$\mathbf{v}=(0,0)$};
\tikzstyle{every node}+=[fill,circle, inner sep=0, minimum size=4pt]
\node(a11) at (1,1) {};
\node(a12) at (1,2) {};
\node(a13) at (1,3) {};
\node(a21) at (2,1) {};
\node(a22) at (2,2) {};
\node(a31) at (3,1) {};

\end{scope}

\begin{scope}[xshift=110]
%\node[scale=1] at (2,0){$\mathbf{v}=(1,0)$};

\tikzstyle{every node}+=[fill,circle, inner sep=0, minimum size=4pt]
\node(a11) at (1,1) {};
\node(a12) at (1,2) {};
\node(a13) at (1,3) {};
\node(a21) at (2,1) {};
\node(a22) at (2,2) {};
\node(a31) at (3,1) {};
\draw (a21)--(a31);
\end{scope}

\begin{scope}[xshift=220]
%\node[scale=1] at (2,0){$\mathbf{v}=(1,1)$};

\tikzstyle{every node}+=[fill,circle, inner sep=0, minimum size=4pt]
\node(a11) at (1,1) {};
\node(a12) at (1,2) {};
\node(a13) at (1,3) {};
\node(a21) at (2,1) {};
\node(a22) at (2,2) {};
\node(a31) at (3,1) {};
\draw (a21)--(a31);
\draw (a11)--(a22);
\end{scope}

\begin{scope}[xshift=330]
%\node[scale=1] at (2,0){$\mathbf{v}=(0,1)$};

\tikzstyle{every node}+=[fill,circle, inner sep=0, minimum size=4pt]
\node(a11) at (1,1) {};
\node(a12) at (1,2) {};
\node(a13) at (1,3) {};
\node(a21) at (2,1) {};
\node(a22) at (2,2) {};
\node(a31) at (3,1) {};
\draw (a11)--(a21);
\end{scope}

\begin{scope}[xshift=440]
%\node[scale=1] at (2,0){$\mathbf{v}=(0,2)$};

\tikzstyle{every node}+=[fill,circle, inner sep=0, minimum size=4pt]
\node(a11) at (1,1) {};
\node(a12) at (1,2) {};
\node(a13) at (1,3) {};
\node(a21) at (2,1) {};
\node(a22) at (2,2) {};
\node(a31) at (3,1) {};
\draw (a11)--(a21);
\draw (a12)--(a22);
\end{scope}
\end{tikzpicture}
    \caption{The set of gravity diagrams in $\GD_{\Zig_5}(3,-1,-1,-1,0)$.}
    \label{figure:zigzag_gravity}
\end{figure}

Using gravity diagrams for the zigzag graph, we obtain a new combinatorial proof of Equation~\eqref{equation:Eulernumber},

The {\em Entringer} numbers $E_{n,k}$ are the entries of the {\em Euler--Bernoulli Triangle}~\cite[\href{https://oeis.org/A008282}{A008282}]{OEIS}, the first few rows of which are reproduced in Table~\ref{table:eulerbernoulli} in the Appendix.
The $(n,k)$-th entry in the triangle is the number of down-up alternating permutations of $n+1$ beginning with $k+1$. 
These numbers satisfy the recurrence equation
\begin{equation}\label{eq.entringerrec}
E_{n,k}= \sum_{i=0}^{k-1}E_{n-1,n-1-i}, \qquad\hbox{for } 1\leq k \leq n,
\end{equation}
with initial conditions $E_{n,0}=0$ for all $n\geq0$ and $E_{1,1}=1$.

\begin{proposition}\label{prop.zigzagrecurrence}
Let $\GD_{\Zig_n}(\vc',k)$ denote the subset of gravity diagrams whose first column is incident to exactly $k$ lines, for $k=0,\ldots, n-2$.  Then $\big\lvert\GD_{\Zig_n}(\vc',k)\big\rvert=E_{n-2,n-2-k}$.
\end{proposition}
\begin{proof}
We will prove that the number of gravity diagrams of $\Zig_{n}$ with a specified number of lines incident to the first column satisfies the same recurrence equation as the Entringer numbers.  First, the gravity diagram for $\Zig_3$ is a single dot, so $\big\lvert\GD_{\Zig_{3}}(\vc',0)\big\rvert=1$, which is equal to $E_{1,1}$. Also for all $n$, there is no gravity diagram when $k=n-2$, since the second column of these diagrams has only $n-3$ dots.  We conclude $\big\lvert\GD_{\Zig_{n}}(\vc',n-2)\big\rvert=0$, which matches $E_{n-2,0}$.

Now fix $1\leq k\leq n-3$.  Since there are $k$ lines incident to dots in the first column of a diagram $D$ in $\GD_{\Zig_{n}}(\vc',k)$, there are at most $n-3-k$ dots in the second column of $D$ which may be connected to dots in the third column of $D$, and the restriction of $D$ to its last $n-3$ columns yields a diagram in \[\bigcup_{j=0}^{n-3-k} \GD_{\Zig_{n-1}}(\vc',j).\]  (We note that for $\Zig_{n}$, $\vc' = (n-2)\alpha_1+ \cdots + \alpha_{n-2}$, while for $\Zig_{n-1}$, $\vc' = (n-3)\alpha_1 + \cdots + \alpha_{n-3}$.) On the other hand, any diagram \[D'\in \bigcup_{j=0}^{n-3-k} \GD_{\Zig_{n-1}}(\vc',j)\] can be uniquely augmented to obtain a diagram in $\GD_{\Zig_{n}}(\vc',k)$ by adding a column of $n-2$ dots to the left of $D'$, and then joining $k$ pairs of dots between the two leftmost columns of the augmented diagram, subject to the gravity rules for $\Zig_{n}$.  Therefore,
$$\big\lvert\GD_{\Zig_{n}}(\vc',k)\big\rvert=\sum_{j=0}^{n-3-k}\big\lvert \GD_{\Zig_{n-1}}(\vc',j)\big\rvert,$$
which is the same recurrence equation as the one for the Entringer numbers. 
\end{proof}

The {\em \mbox{$n$-th} Euler number $E_n$}~\cite[\href{https://oeis.org/A000111}{A000111}]{OEIS} is the diagonal entry $E_{n,n}$ in the Entringer--Bernoulli Triangle. Using Proposition~\ref{prop.zigzagrecurrence} we derive the following.

\begin{proposition}\label{prop:Euler} The volume of the flow polytope $\vol \F_{\Zig_{n+1}}(1,0,\dots,0)$ is $E_{n-1}$.
\end{proposition}
\begin{proof} The volume of $\F_{\Zig_{n+1}}(1,0,\dots,0)$ is the number of gravity diagrams 
\[\big\lvert\GD_{\Zig_{n}}(\vc')\big\rvert 
= \sum_{k=0}^{n-2} \big\lvert\GD_{\Zig_{n}}(\vc',k)\big\rvert
= \sum_{k=0}^{n-2} E_{n-2,n-2-k}
= E_{n-1,n-1}.\qedhere\]
\end{proof}

%%%%%%%%%%%%%%%%%%%%%%%%%%%%%%
\section{Reinterpreting the Lidskii volume formula} \label{section:LidskiiInterp}

In this section, we provide a new combinatorial interpretation of the Lidskii volume formula.  We build upon gravity diagrams and parking functions to define unified diagrams which give a combinatorial interpretation of the right hand side of Equation~\eqref{eq:lidskiivol}. 

%%%%%%%%%%%%%%%%%%%%%%%%%%%%%%
\subsection{Dyck paths as weak compositions} \label{section:tDyck}
From this point forward, our convention will be that a {\em Dyck path} is a lattice path from $(0,0)$ to $(n,n)$ comprised of north steps $N=(0,1)$ and east steps $E=(1,0)$ that stays above the diagonal line from $(0,0)$ to $(n,n)$. One way to represent a Dyck path is as an $\{N,E\}$-word of length $2n$ representing the sequence of steps taken by the path.  We can write this word in the form  $N^{s_1}EN^{s_2}E\cdots N^{s_n}E$ where $s_i$ is the (nonnegative) number of north steps $N$ with $x$-coordinate $i-1$, from which we see that the sequence $\vs=(s_1,s_2,\cdots,s_n)$ is a weak composition of $n$.  We will not make any distinction between the Dyck path $N^{s_1}EN^{s_2}E\cdots N^{s_n}E$ and the weak composition $\vs$ that represents it. 

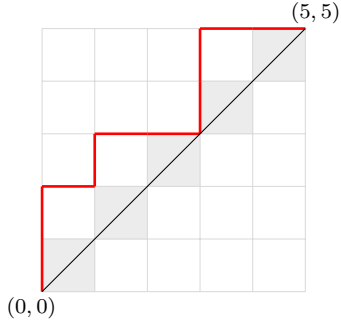
\begin{figure}
    \centering
    \begin{tikzpicture}[scale=0.7]
%define  the parameters
\edef \n{5}% This is the number of columns in the board
\edef \t{{1,1,1,1,1}}% This is the vector \t that gives the board
\edef \isdyck{1}%0 if you don't want a Dyck path 1 otherwise
\edef \pa{{2,1,0,2,0}}%  This is the Dyck path in the board
\edef \isparking{0}%0 if you don't want a parking function 1 otherwise
\edef \pf{{2,3,4,5,6,1}} % This is the parking function decoration
\edef \isgravity{0}%0 if you don't want the gravity diagram 1 otherwise, however you have to  manually edit the gravity diagram at the end of this group
\edef \isdiagonal{1}%0 if you don't want the diagonal 1 otherwise

\node at (-0.2,-0.3) {\tiny$(0,0)$};
\node at (5.2,5.3) {\tiny$(5,5)$};

% Here comes the construction of the board
\pgfmathparse{\n-1}
\global\let\nminus\pgfmathresult

\edef\col{0}
\edef\row{0}

\foreach \x in {0,1,...,\nminus}{
\pgfmathparse{\t[\x]}
\global\let\part\pgfmathresult
\pgfmathsetmacro \newrow {\row+\part}
\pgfmathsetmacro \newcol {\col+1}
\draw[fill, color=gray!15,] (\col,\row) rectangle (\newcol,\newrow);
\global\let\col\newcol
\global\let\row\newrow
}
\draw[very thin, color=gray!30] (0, 0) grid (\col, \row);

%draw path and parking decoration

\if \isdyck0
\else
\pgfmathsetmacro \colpa{0}
\pgfmathsetmacro \rowpa{0}
\pgfmathsetmacro \newrowpa{0}
\pgfmathsetmacro \newcolpa{0}

\foreach \x in {0,1,...,\nminus}{
\pgfmathsetmacro \part {\pa[\x]}

 \if \part0
\else
\foreach \i in {1,...,\part}{
\pgfmathsetmacro \newrowpa {\rowpa+1}
\pgfmathsetmacro \lbpart {\pf[\rowpa]}
\if \isparking0]
\draw [line width=1, color=red] (\colpa, \rowpa)--(\newcolpa,\newrowpa);
\else
\draw [line width=1, color=red] (\colpa, \rowpa)--(\newcolpa,\newrowpa)node [midway, left] {\lbpart};
\fi
\global\let\rowpa\newrowpa
\global\let\newrowpa\rowpa
}
\fi
\pgfmathsetmacro \newcolpa{\colpa+1}
\draw [line width=1, color=red] (\colpa, \rowpa)--(\newcolpa,\newrowpa);
\global\let\colpa\newcolpa
\global\let\newcolpa\colpa
}
\fi

%draw diagonal
 \if \isdiagonal0
\else
\draw (0,0)--(\col,\row);
\fi

%Draw gravity digram
% Write the coordinates of the nodes  you want and lines between them
\if \isgravity0
\else
\tikzstyle{every node}=[circle, draw, fill,inner sep=0pt, minimum width=4pt,scale=1]
\node at (0.5,3.5){};
\node at (0.5,4.5){};
\node at (1.5,4.5){};
\node at (2.5,5.5){};
\draw[line width=1] (0.5,3.5) -- (2.5,5.5);
\fi

\end{tikzpicture}
    \caption{We write the Dyck path $NNENEENNEE$ as the composition $\vs=(2,1,0,2,0)$ of north steps between consecutive east steps.  The shaded boxes along the diagonal correspond to the composition $\vt=(1,1,1,1,1)$.}
    \label{figure:example_dyck_path}
\end{figure}

The set of Dyck paths correspond exactly to those weak compositions $\vs$ that dominate $\vt=(1,1,\ldots,1)$, the Dyck path closest to the diagonal. Figure~\ref{figure:example_dyck_path} provides an example where the highlighted Dyck path $\vs=(2,1,0,2,0)$ dominates $(1,1,1,1,1)$, the latter being indicated by shaded boxes.  This leads to a natural generalization of Dyck paths.  For an arbitrary weak composition $\vt$, we define {\em $\vt$-Dyck paths} to be weak compositions $\vs$ such that $\vs \rhd \vt$.  In a forthcoming paper \cite{CeballosDleon2018}, Ceballos and Gonz\'alez D'Le\'on study generalizations of Catalan combinatorics in the context of $\vt$-Dyck Paths. Figure \ref{figure:example_t_dyck_path} illustrates an example of a $(3,0,2,1,1,0)$-Dyck path. Note that $\vt$ completely describes the rectangular lattice in which a $\vt$-Dyck path lives; the number of columns of boxes is exactly the number of entries in $\vt$ and the number of rows of boxes is $|\vt|$.

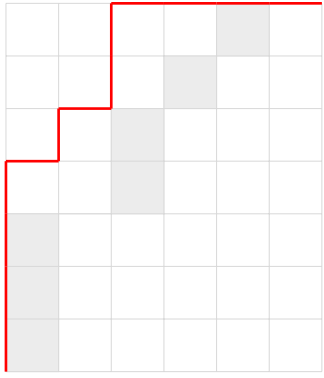
\begin{figure}
    \centering
    \begin{tikzpicture}[scale=0.7]
%define  the parameters
\edef \n{6}% This is the number of columns in the board
\edef \t{{3,0,2,1,1,0}}% This is the vector \t that gives the board
\edef \isdyck{1}%0 if you don't want a Dyck path 1 otherwise
\edef \pa{{4,1,2,0,0,0}}%  This is the Dyck path in the board
\edef \isparking{0}%0 if you don't want a parking function 1 otherwise
\edef \pf{{2,3,4,5,6,1,7}} % This is the parking function decoration
\edef \isgravity{0}%0 if you don't want the gravity diagram 1 otherwise, however you have to  manually edit the gravity diagram at the end of this group
\edef \isdiagonal{0}%0 if you don't want the diagonal 1 otherwise

% Here comes the construction of the board
\pgfmathparse{\n-1}
\global\let\nminus\pgfmathresult

\edef\col{0}
\edef\row{0}

\foreach \x in {0,1,...,\nminus}{
\pgfmathparse{\t[\x]}
\global\let\part\pgfmathresult
\pgfmathsetmacro \newrow {\row+\part}
\pgfmathsetmacro \newcol {\col+1}
\draw[fill, color=gray!15,] (\col,\row) rectangle (\newcol,\newrow);
\global\let\col\newcol
\global\let\row\newrow
}
\draw[very thin, color=gray!30] (0, 0) grid (\col, \row);

%draw path and parking decoration

\if \isdyck0
\else
\pgfmathsetmacro \colpa{0}
\pgfmathsetmacro \rowpa{0}
\pgfmathsetmacro \newrowpa{0}
\pgfmathsetmacro \newcolpa{0}

\foreach \x in {0,1,...,\nminus}{
\pgfmathsetmacro \part {\pa[\x]}

 \if \part0
\else
\foreach \i in {1,...,\part}{
\pgfmathsetmacro \newrowpa {\rowpa+1}
\if \isparking0]
\draw [line width=1, color=red] (\colpa, \rowpa)--(\newcolpa,\newrowpa);
\else
\pgfmathsetmacro \lbpart {\pf[\rowpa]}
\draw [line width=1, color=red] (\colpa, \rowpa)--(\newcolpa,\newrowpa)node [midway, left] {\lbpart};
\fi
\global\let\rowpa\newrowpa
\global\let\newrowpa\rowpa
}
\fi
\pgfmathsetmacro \newcolpa{\colpa+1}
\draw [line width=1, color=red] (\colpa, \rowpa)--(\newcolpa,\newrowpa);
\global\let\colpa\newcolpa
\global\let\newcolpa\colpa
}
\fi

%draw diagonal
 \if \isdiagonal0
\else
\draw (0,0)--(\col,\row);
\fi

%Draw gravity digram
% Write the coordinates of the nodes  you want and lines between them
\if \isgravity0
\else
\tikzstyle{every node}=[circle, draw, fill,inner sep=0pt, minimum width=4pt,scale=1]
\node at (0.5,3.5){};
\node at (0.5,4.5){};
\node at (1.5,4.5){};
\node at (2.5,5.5){};
\draw[line width=1] (0.5,3.5) -- (2.5,5.5);
\fi

\end{tikzpicture}
    \caption{When $\vt=(3,0,2,1,1,0)$, the composition $\vs=(4,1,2,0,0,0)$ is a $\vt$-Dyck path since $\vt$ and $\vs$ have the same length, $|\vt|=|\vs|$, and $\vs\rhd\vt$. }
    \label{figure:example_t_dyck_path}
\end{figure}

%%%%%%%%%%%%%%%%%%%%%%%%%%%%%%
\subsection{Generalized Parking functions} \label{section:parkingfunctions}

A {\em labeled Dyck path} is a Dyck path from $(0,0)$ to $(n,n)$ in which the north steps are labeled by a permutation of $[n]$ in such a way that consecutive north steps are labeled by increasing values.  In other words, a labeled Dyck path is a pair $(\vs,\sigma)$ where $\vs$ is a Dyck path and $\sigma$ is a permutation of $[n]$ whose descents occur only in positions $s_1+s_2+\cdots+s_i$ for $1\leq i\leq n-1$.  An example of a labeled Dyck path is presented in Figure~\ref{figure:example_parking_function}.

Labeled Dyck paths are in bijection with another famous family of combinatorial objects known as parking functions.  (See, for example, Haglund~\cite[Proposition 5.0.1]{Haglund2008}.)  A {\em parking function} of $n$ is a sequence of $n$ positive integers $\vp=(p_1,p_2,\dots,p_n)$ with the property that when the sequence is rearranged in weakly increasing order, the $i$-th entry is less than or equal to $i$. We denote by $\PF_n$ the set of parking functions of $n$.

A different but equivalent description of the vectors $\vp$ explains the name parking function, and was given by Konheim and Weiss~\cite{KonheimWeiss1966}. 
Consider the scenario where we have $n$ distinct vehicles that want to park in $n$ marked parking spaces on a one-way street. Car $\textup{\small \faCar\,}_i$ prefers to park in space $p_i$ so they drive down the street until it reaches its preferred space.  Car $\textup{\small \faCar\,}_i$ will park in space $p_i$ if it is available.  Otherwise, it will continue down the street and park in the next available space.  Given a parking preference vector $\vp=(p_1,p_2,\dots,p_n)$, every car will be able to successfully park if and only if $\vp$ is a parking function. 

A well-known bijection between labeled Dyck paths and parking functions is as follows. 
Given a labeled Dyck path $(\vs,\sigma)$, define the corresponding preference vector $\vp=(p_1,\ldots,p_n)$ by setting $p_i$ equal to one more than the $x$-coordinate of the north step labeled by $i$.
The inverse map takes a parking function preference vector $\vp$ and creates a labeled Dyck path $(\vs,\sigma)$ where $s_i$ is the number of occurrences of $i$ in $\vp$ and $\sigma$ is the permutation whose entries $s_1+\cdots +s_{i-1}+1$ through $s_1+\cdots+s_{i}$ are the positions in $\vp$ that contain an $i$, written in increasing order.  Due to this simple bijection we  consider labeled Dyck paths and parking functions as two different descriptions of the same family of objects. It is known that the number of parking functions, and hence the number of labeled Dyck paths, of $n$ is $(n+1)^{n-1}$. 

\begin{example}\label{ex:ParkingBijection}
For the labeled Dyck path in Figure~\ref{figure:example_parking_function}, we see that labels 2 and 5 are in the first column, so the second and fifth entries of $\vp$ are 1.  Similarly, the third entry of $\vp$ is 2, and the first and fourth entries of $\vp$ are 4.  The inverse of this bijection can be read by noticing that $s_1=2$ since two entries of $\vp$ are 1, $s_2=1$ since one entry of $\vp$ is $2$, $s_4=2$ since two entries of $\vp$ are 4, and all other $s_i=0$.  By inserting the corresponding entries as labels in increasing order, we arrive at the labeled Dyck path.  The parking preference vector $\vp=(4,1,2,4,1)$ yields the depicted successful parking configuration by the algorithm described above. 
\end{example}

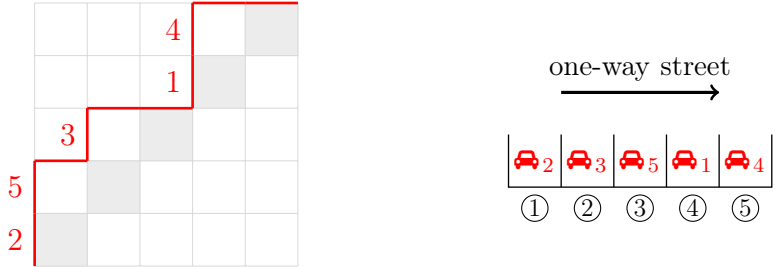
\begin{figure}[htb]
    \centering
    \begin{tikzpicture}[scale=0.7]
%define  the parameters
\edef \n{5}% This is the number of columns in the board
\edef \t{{1,1,1,1,1}}% This is the vector \t that gives the board
\edef \isdyck{1}%0 if you don't want a Dyck path 1 otherwise
\edef \pa{{2,1,0,2,0}}%  This is the Dyck path in the board
\edef \isparking{1}%0 if you don't want a parking function 1 otherwise
\edef \pf{{2,5,3,1,4}} % This is the parking function decoration
\edef \isgravity{0}%0 if you don't want the gravity diagram 1 otherwise, however you have to  manually edit the gravity diagram at the end of this group
\edef \isdiagonal{0}%0 if you don't want the diagonal 1 otherwise

%\node at (-0.2,-0.3) {\tiny$(0,0)$};
%\node at (5.2,5.3) {\tiny$(5,5)$};

% Here comes the construction of the board
\pgfmathparse{\n-1}
\global\let\nminus\pgfmathresult

\edef\col{0}
\edef\row{0}

\foreach \x in {0,1,...,\nminus}{
\pgfmathparse{\t[\x]}
\global\let\part\pgfmathresult
\pgfmathsetmacro \newrow {\row+\part}
\pgfmathsetmacro \newcol {\col+1}
\draw[fill, color=gray!15,] (\col,\row) rectangle (\newcol,\newrow);
\global\let\col\newcol
\global\let\row\newrow
}
\draw[very thin, color=gray!30] (0, 0) grid (\col, \row);

%draw path and parking decoration

\if \isdyck0
\else
\pgfmathsetmacro \colpa{0}
\pgfmathsetmacro \rowpa{0}
\pgfmathsetmacro \newrowpa{0}
\pgfmathsetmacro \newcolpa{0}

\foreach \x in {0,1,...,\nminus}{
\pgfmathsetmacro \part {\pa[\x]}

 \if \part0
\else
\foreach \i in {1,...,\part}{
\pgfmathsetmacro \newrowpa {\rowpa+1}
\pgfmathsetmacro \lbpart {\pf[\rowpa]}
\if \isparking0]
\draw [line width=1, color=red] (\colpa, \rowpa)--(\newcolpa,\newrowpa);
\else
\draw [line width=1, color=red] (\colpa, \rowpa)--(\newcolpa,\newrowpa)node [midway, left] {\lbpart};
\fi
\global\let\rowpa\newrowpa
\global\let\newrowpa\rowpa
}
\fi
\pgfmathsetmacro \newcolpa{\colpa+1}
\draw [line width=1, color=red] (\colpa, \rowpa)--(\newcolpa,\newrowpa);
\global\let\colpa\newcolpa
\global\let\newcolpa\colpa
}
\fi

%draw diagonal
 \if \isdiagonal0
\else
\draw (0,0)--(\col,\row);
\fi

%Draw gravity digram
% Write the coordinates of the nodes  you want and lines between them
\if \isgravity0
\else
\tikzstyle{every node}=[circle, draw, fill,inner sep=0pt, minimum width=4pt,scale=1]
\node at (0.5,3.5){};
\node at (0.5,4.5){};
\node at (1.5,4.5){};
\node at (2.5,5.5){};
\draw[line width=1] (0.5,3.5) -- (2.5,5.5);
\fi

%\node[]  at (8.5,2.5) {$\vp=(4,1,2,4,1)$};

\draw[->, very thick]  (10,3.3)--(13,3.3) node [midway, above] {\small one-way street};

\tikzstyle{every node}=[circle,  draw, inner sep=1pt, minimum width=4pt,scale=0.8]

\draw[line width=0.5] (9,2.5) -- (9,1.5)--(14,1.5)--(14,2.5);
\draw[line width=0.5] (10,2.5)--(10,1.5);
\draw[line width=0.5] (11,2.5)--(11,1.5);
\draw[line width=0.5] (12,2.5)--(12,1.5);
\draw[line width=0.5] (13,2.5)--(13,1.5);
\node[]  at (9.5,1.1) {$1$};
\node[]  at (10.5,1.1) {$2$};
\node[]  at (11.5,1.1) {$3$};
\node[]  at (12.5,1.1) {$4$};
\node[]  at (13.5,1.1) {$5$};

\tikzstyle{every node}=[inner sep=1pt,color=red, minimum width=4pt,scale=1]

\node[]  at (9.5,2) {$\textup{\tiny\faCar\,}_2$};
\node[]  at (10.5,2) {$\textup{\tiny\faCar\,}_3$};
\node[]  at (11.5,2) {$\textup{\tiny\faCar\,}_5$};
\node[]  at (12.5,2) {$\textup{\tiny\faCar\,}_1$};
\node[]  at (13.5,2) {$\textup{\tiny\faCar\,}_4$};

\end{tikzpicture}
    \caption{The labeled Dyck path $\big((2,1,0,2,0),25314\big)$ 
    corresponds to the parking preference vector $\vp=(4,1,2,4,1)$.  The final parking arrangement for $\vp$ is on the right. The bijection is discussed in Example~\ref{ex:ParkingBijection}.}
    \label{figure:example_parking_function}
\end{figure}

We can also define labeled $\vt$-Dyck paths $(\vs,\sigma)$ where the north steps of $\vs$ are labeled by a permutation of $\big[\lvert\vt\rvert\big]$.   Labeled $\vt$-Dyck paths are known as {\em generalized parking functions}, and we denote this set by $\PF_{\vt}$.  For certain values of $\vt$, there are nice formulas for $\lvert \PF_{\vt}\rvert$ (see, for example, \cite{Yan2000, Yan2001}), but their enumeration in general is less straightforward.  We can learn more by refining the set $\PF_{\vt}$ by the underlying path $\vs$.

\begin{lemma}
The number of labeled $\vt$-Dyck paths is
\begin{equation}\label{equation:enumeration_allparking}
|\PF_{\vt}|=\sum_{\vs \rhd \vt }\binom{|\vt|}{\vs}.
\end{equation}
\end{lemma}
\begin{proof}
When we consider the subset of
labeled $\vt$-Dyck paths that have underlying path $\vs$, the entries of $\vs$ determine the sequences of consecutive $N$ steps, whose labels must be strictly increasing. The number of labeled $\vt$-Dyck paths with underlying path $\vs$ is \(\binom{|\vt|}{\vs}\).  Summing over all possible paths $\vs$ gives Equation~\eqref{equation:enumeration_allparking}.
\end{proof}

When $\vt=(1,1,\dots,1)$, we recover the following identity for classical parking functions:   
\begin{equation*}\label{equation:tree_number}
(n+1)^{n-1}=\sum_{\vs \rhd (1,1,\dots,1)}\binom{n}{\vs}.
\end{equation*}

The similarity between the sums in Equations~\eqref{eq:lidskiivol} and \eqref{equation:enumeration_allparking} allows us to state a parking function version for the Lidskii volume formula in Theorem~\ref{thm.Lidskii} as a sum over labeled $\vt$-Dyck paths.  This leads naturally to the definition of unified diagrams presented in the next subsection.

\begin{theorem}[Parking function version of the Lidskii volume formula]
\label{thm.Lidskii_version_2}

Let $G$ be a directed graph on $n+1$ vertices with shifted out-degree vector $\vt=(t_1,\ldots, t_n)$ and let $\va = (a_1,\ldots,a_n)$ be a nonnegative integer net flow vector.  Then
\begin{align} \label{eq:lidskiivol_version_2}
\vol\mathcal{F}_G(\va)
= \sum_{(\vs,\sigma) \in\PF_{\vt}} \va^{\vs} \cdot 
K_{G|_n}(\vs-\vt),
\end{align} 
where the sum is over all labeled $\vt$-Dyck paths $(\vs,\sigma)$.
\end{theorem}

%%%%%%%%%%%%%%%%%%%%%%%%%%%%%%
\subsection{Unified diagrams}  \label{section:unifieddiagrams}
We combine generalized parking functions and gravity diagrams to create a new and more general family of combinatorial diagrams that can be used to compute $\vol \mathcal{F}_G(\va)$. Equation~\eqref{eq:lidskiivol_version_2} leads us to consider tuples $(\vs,\sigma,\vphi,D)$ where $(\vs,\sigma)$ is a labeled $\vt$-Dyck path, $\vphi$ is a vector in $[a_1]^{s_1}\times\cdots\times[a_n]^{s_n}\subset \ZZ_{>0}^{|\vs|}$, and $D$ is a gravity diagram in  $\GD_{G|_n}(\vs-\vt)$. 

We wish to define a unified combinatorial object that includes all of this information.  To do so, we supplement the parking function labels given by $\sigma$ by new labels we call {\em net flow labels}, by placing independently a number from $\{1,\ldots, a_i\}$ on each of the north steps with $x$-coordinate $i-1$.  (Repetition of net flow labels is permitted.)  Furthermore, since $\vs$ is a $\vt$-Dyck path, we have that $\vs \rhd \vt$ and that the quantity $\vs-\vt$ expressed in the $\valpha$ basis is
\begin{align}
    \vs-\vt&=(s_1-t_1)\valpha_1+ (s_1-t_1+s_2-t_2)\valpha_2+\cdots+(s_1-t_1+\cdots +s_{n-1}-t_{n-1})\valpha_{n-1}\nonumber\\
    &=\sum_{d=1}^{n-1}\left (\sum_{k=1}^d s_k - t_k \right ) \valpha_d.
\end{align}
Hence, a gravity diagram in $\GD_{G|_n}(\vs-\vt)$ has $\sum_{k=1}^d s_k - t_k$ dots in column $d$, which is exactly the number of cells in column $d$ between the lattice paths $\vs$ and $\vt$.For this reason we can insert our gravity diagram into the diagram under $\vs$ and above $\vt$.  

A {\em unified diagram} for a directed graph $G$ with shifted out-degree vector $\vt$ and a net flow vector $\va$ is a $\vt$-Dyck path $\vs$ whose north steps are labeled both by parking function labels given by a permutation $\sigma$ of $\big[|\vt|\big]=[m-n]$ and by net flow labels given by $\vphi$ in $[a_1]^{s_1}\times\cdots\times[a_n]^{s_n}$ along with a canonical gravity diagram representative embedded in the boxes lying between $\vt$ and $\vs$. We let $\U_G(\va)$ denote the set of such unified diagrams. 

By appealing to the right hand side of Equation~\eqref{eq:lidskiivol_version_2}, we see that unified diagrams give a combinatorial interpretation of $\vol \mathcal{F}_G(\va)$.  The sum is over labeled $\vt$-Dyck paths $(\vs,\sigma)$ while each summand corresponds to the possible parking labels that can be assigned to its north steps and gravity diagrams that can be embedded therein.  By construction we have the following result.

\begin{theorem} \label{theorem:unifieddiagrams} 
For any graph $G$ on $n+1$ vertices and for any nonnegative net flow vector $\va\in\mathbb{Z}^{n}$,
$$\vol \mathcal{F}_G(\va) = \big\lvert \U_G(\va)\big\rvert.$$
\end{theorem}

\begin{example}
\label{ex:unified}
For the graph $G=\Car_6$ and net flow vector $\va=(4,1,3,1,1)$, the shifted out-degree vector is $\vt=(3,1,1,1,0)$. In Figure~\ref{figure:example_caracol_5_11111} we illustrate some of the diagrams in $\U_{\Car_6}(4,1,3,1,1)$. The parking function labels are written to the left of the $\vt$-Dyck path in red and the net flow labels are written in blue to its right. Note for example that since $a_1=4$, the north steps with $x$-coordinate $0$ only have net flow labels with values in $[4]$. Note also that the embedded gravity diagrams have dots below the decorated $\vt$-Dyck path and above the shaded boxes determined by $\vt$. For the caracol graph, we now use the convention that the gravity diagram which is embedded into the unified diagram is sheared by $45$ degrees. 
\end{example}

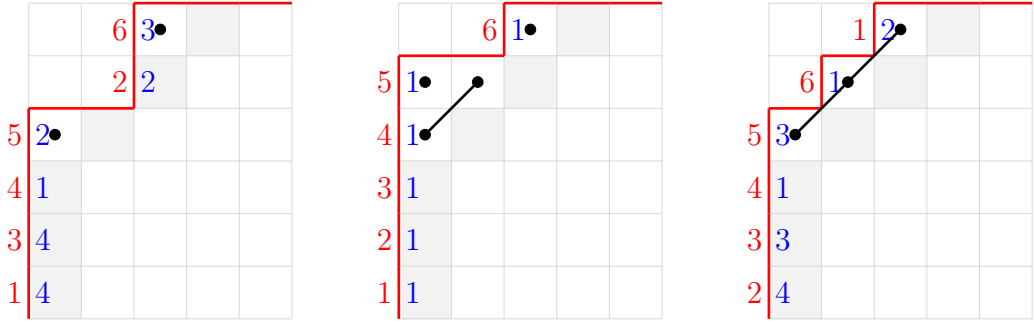
\begin{figure}
    \centering
    \begin{tikzpicture}[scale=0.7]
\begin{scope}[xshift=0]
%define  the parameters
\edef \n{5}% This is the number of columns in the board
\edef \t{{3,1,1,1,0}}% This is the vector \t that gives the board
\edef \isdyck{1}%0 if you don't want a Dyck path 1 otherwise
\edef \pa{{4,0,2,0,0}}%  This is the Dyck path in the board
\edef \isparking{1}%0 if you don't want a parking function 1 otherwise
\edef \pf{{1,3,4,5,2,6}} % This is the parking function decoration
\edef \isflow{1}%0 if you don't want a net flow weights 1 otherwise
\edef \nf{{4,4,1,2,2,3}} % This is the net-flow decoration
\edef \isgravity{1}%0 if you don't want the gravity diagram 1 otherwise, however you have to  manually edit the gravity diagram at the end of this group
\edef \isdiagonal{0}%0 if you don't want the diagonal 1 otherwise

% Here comes the construction of the board
\pgfmathparse{\n-1}
\global\let\nminus\pgfmathresult

\edef\col{0}
\edef\row{0}

\foreach \x in {0,1,...,\nminus}{
\pgfmathparse{\t[\x]}
\global\let\part\pgfmathresult
\pgfmathsetmacro \newrow {\row+\part}
\pgfmathsetmacro \newcol {\col+1}
\draw[fill, color=gray!10,] (\col,\row) rectangle (\newcol,\newrow);
\global\let\col\newcol
\global\let\row\newrow
}
\draw[very thin, color=gray!30] (0, 0) grid (\col, \row);

%draw path and parking decoration

\if \isdyck0
\else
\pgfmathsetmacro \colpa{0}
\pgfmathsetmacro \rowpa{0}
\pgfmathsetmacro \newrowpa{0}
\pgfmathsetmacro \newcolpa{0}

\foreach \x in {0,1,...,\nminus}{
\pgfmathsetmacro \part {\pa[\x]}

 \if \part0
\else
\foreach \i in {1,...,\part}{
\pgfmathsetmacro \newrowpa {\rowpa+1}
\pgfmathsetmacro \lbpart {\pf[\rowpa]}
\pgfmathsetmacro \flpart {\nf[\rowpa]}
\if \isparking0
\if \isflow0
\draw [line width=1, color=red] (\colpa, \rowpa)--(\newcolpa,\newrowpa);
\else
\draw [line width=1, color=red] (\colpa, \rowpa)--(\newcolpa,\newrowpa)node [midway, left] {\flpart};
\fi
\else
\if \isflow0
\draw [line width=1, color=red] (\colpa, \rowpa)--(\newcolpa,\newrowpa)node [midway, left] {\lbpart};
\else
\draw [line width=1, color=red] (\colpa, \rowpa)--(\newcolpa,\newrowpa)node [pos=0.5, left=-2] {\lbpart}node [pos=0.5,right=-2,blue] {\flpart};
\fi
\fi
\global\let\rowpa\newrowpa
\global\let\newrowpa\rowpa
}
\fi
\pgfmathsetmacro \newcolpa{\colpa+1}
\draw [line width=1, color=red] (\colpa, \rowpa)--(\newcolpa,\newrowpa);
\global\let\colpa\newcolpa
\global\let\newcolpa\colpa
}
\fi

%draw diagonal
 \if \isdiagonal0
\else
\draw (0,0)--(\col,\row);
\fi

%Draw gravity diagram
% Write the coordinates of the nodes  you want and lines between them
\if \isgravity0
\else
\tikzstyle{every node}=[circle, draw, fill,inner sep=0pt, minimum width=4pt,scale=1]
\node at (0.5,3.5){};
\node at (2.5,5.5){};
%\draw[line width=1] (0.5,3.5) -- (2.5,5.5);
\fi
\end{scope}

\begin{scope}[xshift=200]
%define  the parameters
\edef \n{5}% This is the number of columns in the board
\edef \t{{3,1,1,1,0}}% This is the vector \t that gives the board
\edef \isdyck{1}%0 if you don't want a Dyck path 1 otherwise
\edef \pa{{5,0,1,0,0}}%  This is the Dyck path in the board
\edef \isparking{1}%0 if you don't want a parking function 1 otherwise
\edef \pf{{1,2,3,4,5,6}} % This is the parking function decoration
\edef \isflow{1}%0 if you don't want a net flow weights 1 otherwise
\edef \nf{{1,1,1,1,1,1}} % This is the net-flow decoration
\edef \isgravity{1}%0 if you don't want the gravity diagram 1 otherwise, however you have to  manually edit the gravity diagram at the end of this group
\edef \isdiagonal{0}%0 if you don't want the diagonal 1 otherwise

% Here comes the construction of the board
\pgfmathparse{\n-1}
\global\let\nminus\pgfmathresult

\edef\col{0}
\edef\row{0}

\foreach \x in {0,1,...,\nminus}{
\pgfmathparse{\t[\x]}
\global\let\part\pgfmathresult
\pgfmathsetmacro \newrow {\row+\part}
\pgfmathsetmacro \newcol {\col+1}
\draw[fill, color=gray!10,] (\col,\row) rectangle (\newcol,\newrow);
\global\let\col\newcol
\global\let\row\newrow
}
\draw[very thin, color=gray!30] (0, 0) grid (\col, \row);

%draw path and parking decoration

\if \isdyck0
\else
\pgfmathsetmacro \colpa{0}
\pgfmathsetmacro \rowpa{0}
\pgfmathsetmacro \newrowpa{0}
\pgfmathsetmacro \newcolpa{0}

\foreach \x in {0,1,...,\nminus}{
\pgfmathsetmacro \part {\pa[\x]}

 \if \part0
\else
\foreach \i in {1,...,\part}{
\pgfmathsetmacro \newrowpa {\rowpa+1}
\pgfmathsetmacro \lbpart {\pf[\rowpa]}
\pgfmathsetmacro \flpart {\nf[\rowpa]}
\if \isparking0
\if \isflow0
\draw [line width=1, color=red] (\colpa, \rowpa)--(\newcolpa,\newrowpa);
\else
\draw [line width=1, color=red] (\colpa, \rowpa)--(\newcolpa,\newrowpa)node [midway, left] {\flpart};
\fi
\else
\if \isflow0
\draw [line width=1, color=red] (\colpa, \rowpa)--(\newcolpa,\newrowpa)node [midway, left] {\lbpart};
\else
\draw [line width=1, color=red] (\colpa, \rowpa)--(\newcolpa,\newrowpa)node [pos=0.5, left=-2] {\lbpart}node [pos=0.5,right=-2,blue] {\flpart};
\fi
\fi
\global\let\rowpa\newrowpa
\global\let\newrowpa\rowpa
}
\fi
\pgfmathsetmacro \newcolpa{\colpa+1}
\draw [line width=1, color=red] (\colpa, \rowpa)--(\newcolpa,\newrowpa);
\global\let\colpa\newcolpa
\global\let\newcolpa\colpa
}
\fi

%draw diagonal
 \if \isdiagonal0
\else
\draw (0,0)--(\col,\row);
\fi

%Draw gravity digram
% Write the coordinates of the nodes  you want and lines between them
\if \isgravity0
\else
\tikzstyle{every node}=[circle, draw, fill,inner sep=0pt, minimum width=4pt,scale=1]
\node at (0.5,3.5){};
\node at (0.5,4.5){};
\node at (1.5,4.5){};
\node at (2.5,5.5){};
\draw[line width=1] (0.5,3.5) -- (1.5,4.5);
\fi
\end{scope}
\begin{scope}[xshift=400]
%define  the parameters
\edef \n{5}% This is the number of columns in the board
\edef \t{{3,1,1,1,0}}% This is the vector \t that gives the board
\edef \isdyck{1}%0 if you don't want a Dyck path 1 otherwise
\edef \pa{{4,1,1,0,0}}%  This is the Dyck path in the board
\edef \isparking{1}%0 if you don't want a parking function 1 otherwise
\edef \pf{{2,3,4,5,6,1}} % This is the parking function decoration
\edef \isflow{1}%0 if you don't want a net flow weights 1 otherwise
\edef \nf{{4,3,1,3,1,2}} % This is the net-flow decoration
\edef \isgravity{1}%0 if you don't want the gravity diagram 1 otherwise, however you have to  manually edit the gravity diagram at the end of this group
\edef \isdiagonal{0}%0 if you don't want the diagonal 1 otherwise

% Here comes the construction of the board
\pgfmathparse{\n-1}
\global\let\nminus\pgfmathresult

\edef\col{0}
\edef\row{0}

\foreach \x in {0,1,...,\nminus}{
\pgfmathparse{\t[\x]}
\global\let\part\pgfmathresult
\pgfmathsetmacro \newrow {\row+\part}
\pgfmathsetmacro \newcol {\col+1}
\draw[fill, color=gray!10,] (\col,\row) rectangle (\newcol,\newrow);
\global\let\col\newcol
\global\let\row\newrow
}
\draw[very thin, color=gray!30] (0, 0) grid (\col, \row);

%draw path and parking decoration

\if \isdyck0
\else
\pgfmathsetmacro \colpa{0}
\pgfmathsetmacro \rowpa{0}
\pgfmathsetmacro \newrowpa{0}
\pgfmathsetmacro \newcolpa{0}

\foreach \x in {0,1,...,\nminus}{
\pgfmathsetmacro \part {\pa[\x]}

 \if \part0
\else
\foreach \i in {1,...,\part}{
\pgfmathsetmacro \newrowpa {\rowpa+1}
\pgfmathsetmacro \lbpart {\pf[\rowpa]}
\pgfmathsetmacro \flpart {\nf[\rowpa]}
\if \isparking0
\if \isflow0
\draw [line width=1, color=red] (\colpa, \rowpa)--(\newcolpa,\newrowpa);
\else
\draw [line width=1, color=red] (\colpa, \rowpa)--(\newcolpa,\newrowpa)node [midway, left] {\flpart};
\fi
\else
\if \isflow0
\draw [line width=1, color=red] (\colpa, \rowpa)--(\newcolpa,\newrowpa)node [midway, left] {\lbpart};
\else
\draw [line width=1, color=red] (\colpa, \rowpa)--(\newcolpa,\newrowpa)node [pos=0.5, left=-2] {\lbpart}node [pos=0.5,right=-2,blue] {\flpart};
\fi
\fi
\global\let\rowpa\newrowpa
\global\let\newrowpa\rowpa
}
\fi
\pgfmathsetmacro \newcolpa{\colpa+1}
\draw [line width=1, color=red] (\colpa, \rowpa)--(\newcolpa,\newrowpa);
\global\let\colpa\newcolpa
\global\let\newcolpa\colpa
}
\fi

%draw diagonal
 \if \isdiagonal0
\else
\draw (0,0)--(\col,\row);
\fi

%Draw gravity diagram
% Write the coordinates of the nodes  you want and lines between them
\if \isgravity0
\else
\tikzstyle{every node}=[circle, draw, fill,inner sep=0pt, minimum width=4pt,scale=1]
\node at (0.5,3.5){};
\node at (1.5,4.5){};
\node at (2.5,5.5){};
\draw[line width=1] (0.5,3.5) -- (2.5,5.5);
\fi
\end{scope}
\end{tikzpicture}
    \caption{Three unified diagrams in $\U_{\Car_6}(4,1,3,1,1)$ as discussed in Example~\ref{ex:unified}.  The parking function labels are written to the left of the $(3,1,1,1,0)$-Dyck path in red and the net flow labels are written to the right in blue.}
    \label{figure:example_caracol_5_11111}
\end{figure}

\begin{remark}\label{remark!}
Note that when $a_i=0$ for some $i$, the only unified diagrams that contribute to the Lidskii sum are those with no north steps with $x$-coordinate $i-1$ because there are no choices for labels on those steps.

Moreover, if the net flow vector is $\va=(1,0,\ldots,0)$ then the only unified diagrams that contribute to the Lidskii sum must have $\vt$-Dyck path $\vs=(m-n,0,\ldots,0)$.  Therefore, each unified diagram is completely characterized by its gravity diagram, which provides a combinatorial explanation for Corollary~\ref{cor:volume}.
\end{remark}

When $\va$ is a zero-one vector, all net flow labels are $1$.  We omit the net flow labels from the unified diagram when this is the case.

%%%%%%%%%%%%%%%%%%%%%%%%%%%%%%
\subsection{Unified diagrams for the Pitman--Stanley graph}
We can apply Theorem~\ref{theorem:unifieddiagrams} to give a new proof of Equation~\eqref{equation:PitmanStanleyvolume}, a classical result of Pitman and Stanley~\cite{PitmanStanley2002}.

\begin{proposition} 
For the Pitman--Stanley graph $\PS_{n+1}$ and net flow vector $\va=(1,\dots,1)$, we have $\vol \F_{\PS_{n+1}}(1,\hdots,1)=n^{n-2}$.
\end{proposition}
\begin{proof}
Given that the shifted out-degree vector of $\PS_{n+1}$ is $\vt=(1,\dots,1,0)$ and any gravity diagram of $\PS_{n+1}$ contains only dots (see Proposition~\ref{prop:GravityPS}), a diagram in $\U_{\PS_{n+1}}(\va)$ is completely characterized by its labeled $\vt$-Dyck path, which is a parking function in $\PF_{n-1}$.  The result follows.
\end{proof}

%%%%%%%%%%%%%%%%%%%%%%%%%%%%%%%%%%%%%%%%%%%%%%%%%%%%%%%%%%%%%
%%%%%%%%%%%%%%%%%%%%%%%%%%%%%%%%%%%%%%%%%%%%%%%%%%%%%%%%%%%%%
\section{The volume of \texorpdfstring{$\F_{\Car_{n+1}}(1,\dots,1)$}{F Car (1,...,1)}} 
\label{section:volumecaracol}

In this section we use Theorem \ref{theorem:unifieddiagrams} to compute the volume of $\F_{\Car_{n+1}}(1,\dots,1)$.

%%%%%%%%%%%%%%%%%%%%%%%%%%%%%%
\subsection{Refined Unified Diagrams}
\label{sec:RUD}
We enumerate unified diagrams by refining this set according to the first east step of the underlying $\vt$-Dyck path. For a unified diagram $(\vs,\sigma,\vphi,D)\in\U_{G}(\va)$, label the horizontal lines of the Dyck path from top to bottom starting with $i=0$, and suppose that the first east step of $\vs$ is along the horizontal line labeled by $i$.  (This labeling scheme is shown in Figure~\ref{figure:example_modified_diagram}; for the example therein, $i=3$.)

\begin{figure}
    \centering
    \begin{tikzpicture}[scale=0.5]
%define  the parameters
\edef \n{7}% This is the number of columns in the board
\edef \t{{5,1,1,1,1,1,0}}% This is the vector \t that gives the board
\edef \isdyck{1}%0 if you don't want a Dyck path 1 otherwise
\edef \pa{{0,2,0,1,0,0,0}}%  This is the Dyck path in the board
\edef \isparking{1}%0 if you don't want a parking function 1 otherwise
\edef \pf{{0,0,0,0,0,0,0,1,3,2}} % This is the parking function decoration with enough zeroes for the first steps
\edef \isgravity{1}%0 if you don't want the gravity diagram 1 otherwise, however you have to  manually edit the gravity diagram at the end of this group
\edef \initial{3}%0 if you don't want the diagonal 1 otherwise
%This is the additional construction

% Here comes the construction of the board
\pgfmathparse{\n-1}
\global\let\nminus\pgfmathresult

\edef\col{0}
\edef\row{0}

\foreach \x in {0,1,...,\nminus}{
\pgfmathparse{\t[\x]}
\global\let\part\pgfmathresult
\pgfmathsetmacro \newrow {\row+\part}
\pgfmathsetmacro \newcol {\col+1}
\draw[fill, color=gray!10,] (\col,\row) rectangle (\newcol,\newrow);
\global\let\col\newcol
\global\let\row\newrow
}
\draw[very thin, color=gray!30] (0, 0) grid (\col, \row);

%Write the values of i
\pgfmathparse{\n-2}
\global\let\nminustwo\pgfmathresult
\pgfmathparse{2*\nminustwo}
\global\let\doublenminustwo\pgfmathresult
\foreach \i in {0,1,...,\doublenminustwo}{

\pgfmathparse{\doublenminustwo-\i}
\node at (-1,\pgfmathresult) {\tiny$i=\i$};
}
%draw path and parking decoration

\if \isdyck0
\else
\pgfmathparse{\doublenminustwo-\initial}
\pgfmathsetmacro \colpa{0}
\pgfmathsetmacro \rowpa{\pgfmathresult}
\pgfmathsetmacro \newrowpa{\pgfmathresult}
\pgfmathsetmacro \newcolpa{0}

\foreach \x in {0,1,...,\nminus}{
\pgfmathsetmacro \part {\pa[\x]}

 \if \part0
\else
\foreach \i in {1,...,\part}{
\pgfmathsetmacro \newrowpa {\rowpa+1}
\if \isparking0
\draw [line width=1, color=red] (\colpa, \rowpa)--(\newcolpa,\newrowpa);
\else
\pgfmathsetmacro \lbpart {\pf[\rowpa]}
\draw [line width=1, color=red] (\colpa, \rowpa)--(\newcolpa,\newrowpa)node [midway, left] {\lbpart};
\fi
\global\let\rowpa\newrowpa
\global\let\newrowpa\rowpa
}
\fi
\pgfmathsetmacro \newcolpa{\colpa+1}
\draw [line width=1, color=red] (\colpa, \rowpa)--(\newcolpa,\newrowpa);
\global\let\colpa\newcolpa
\global\let\newcolpa\colpa
}
\fi

%Draw gravity digram
% Write the coordinates of the nodes  you want and lines between them
\if \isgravity0
\else
\tikzstyle{every node}=[circle, draw, fill,inner sep=0pt, minimum width=4pt,scale=1]
\node at (0.5,5.5){};
\node at (0.5,6.5){};
\node at (1.5,6.5){};
\node at (1.5,7.5){};
\node at (1.5,8.5){};

\node at (2.5,7.5){};
\node at (2.5,8.5){};

\node at (3.5,8.5){};
\node at (3.5,9.5){};

\node at (4.5,9.5){};

\draw[line width=1] (0.5,6.5) -- (1.5,7.5);
\draw[line width=1] (0.5,5.5) -- (3.5,8.5);
\fi

\end{tikzpicture}
    \caption{Example of a level-3 refined unified diagram for $G=\Car_8$ and $\va=(1,\ldots,1)$.}
    \label{figure:example_modified_diagram}
\end{figure}
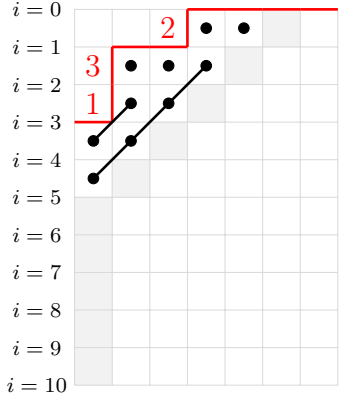

For a fixed $i\geq0$, the set of {\em level-$i$ refined unified diagrams}, denoted $\U^i_{G}(\va)$, is the set of unified diagrams whose north steps along the first column are omitted, and whose first east step is along the horizontal line labeled by $i$. Furthermore, the parking function labels on the remaining north steps of the $\vt$-Dyck path are standardized to lie in the set $[i]$. There is no change to the remaining net flow labels.  Because north steps along the first column are omitted, the first net flow label ($a_1$) has no impact on the set $\U^i_{G}(\va)$.

\begin{proposition}\label{p:iUnified}
Let $G$ be a directed graph with $n+1$ vertices, $m$ edges, and shifted out degree vector $\vt=(t_1,\ldots,t_n)$ and let $\va\in\ZZ_{\geq0}^n$ be any nonnegative net flow vector.  Then
\begin{align}\label{eq:refinedU}
    |\U_{G}(\va)|=\sum_{i=0}^{m-n-t_1}\binom{m-n}{i}a_1^{m-n-i}\big\lvert\U^i_{G}(\va)\big\rvert.
\end{align}
\end{proposition}
\begin{proof}
We condition on the level of the first east step of the unified diagram, which must occur no lower than level $|\vt|-t_1=m-n-t_1$. There are $\binom{m-n}{m-n-i}=\binom{m-n}{i}$ choices for the parking function labels and $a_1^{|\vt|-i}$ choices for the net flow labels on the north steps in the first column of a unified diagram.  The result follows from this.
\end{proof}

\subsection{The parking triangle}

For the remainder of this section, we re-index by setting $r=n-2$ and consider the caracol graph $G=\Car_{n+1}=\Car_{r+3}$.  It has shifted out-degree vector $\vt=(r,1,\ldots,1,0)$ and $|\vt|=2r$.  
Since initial east steps of refined unified diagrams only occur at levels $0$ through $r$, the array of numbers defined by 
\begin{equation}
    T(r,i):=\big\lvert\U^i_{\Car_{r+3}}(1,1,\dots,1)\big\rvert \textup{ for }
    0\leq i\leq r\end{equation}
contains all the information necessary to compute $\vol\mathcal{F}_G(1,1,\ldots,1)$. Indeed, we have the following special case of Proposition~\ref{p:iUnified}.

\begin{corollary}\label{cor:decompositionofunifieddiagrams}
For the caracol graph $G=\Car_{r+3}$ and net flow vector $\va=(1,1,\ldots,1)$, 
\begin{align}
    \big\lvert\U_{\Car_{r+3}}(1,1,\dots,1)\big\rvert=\sum_{i=0}^{r}\binom{2r}{i}T(r,i).
\end{align}
\end{corollary}

We call the array of numbers $T(r,i)$ the {\em parking triangle}.  See Table~\ref{table:parkingtriangle} for a table of values.  The reader will notice the surprising enumerative property that each row of the parking triangle interpolates between the Catalan number $T(r,0)=C_{r}$ (Proposition~\ref{prop:parkingtriangleCatalan}) and the number of parking functions $T(r,r)=(r+1)^{r-1}$ (Proposition~\ref{proposition:parkingtriangleparkingend}).  Theorem~\ref{thm:triangleclosedformulas} proves a general formula for $T(r,i)$, which includes the fact that $T(r,r-1)$ also equals $(r+1)^{r-1}$.

\begin{proposition}\label{prop:parkingtriangleCatalan}
For any $r\ge 0$, $T(r,0)=C_{r}$.
\end{proposition}
\begin{proof} Every level-0 refined unified diagram contains no north steps, so they are counted by the number of gravity diagrams for $(\Car_{r+3})|_{r+2}$.  Proposition~\ref{prop:Dyck1} shows that this number is $C_r$.
\end{proof}

\begin{proposition}\label{proposition:parkingtriangleparkingend}
For any $r\ge 0$, $T(r,r)=(r+1)^{r-1}$.
\end{proposition}
\begin{proof} When the first east step of the Dyck path of a refined unified diagram is at level $r$, the associated gravity diagrams for $(\Car_{r+3})|_{r+2}$ must only consist of dots, because none of the possible first-column-adjusted lines can fit. Therefore, each refined unified diagram is completely determined by its parking function (which is an $r$-parking function), so $T(r,r)=(r+1)^{r-1}$ by Equation~\eqref{equation:tree_number}. 
\end{proof}

\begin{remark}
In future work, we further investigate the properties of the parking triangle numbers $T(r,i)$; we mention some of these forthcoming results now. By partitioning the set $\U^1_{\Car_{n+1}}(1,\hdots,1)$ according to the column containing the unique north step labeled $1$, it can be shown that the numbers $T(r,1)$ satisfy an identity similar to that of the Catalan numbers $T(r,0)=C_r$. Namely,
$$T(r,1) = \sum_{j=1}^r (r+1-j) C_{j-1}C_{r-j}\qquad \hbox{for } r\geq1.$$
Furthermore, by considering the coefficient of $[x^{n-3}]$ in $C(x) \frac{d}{dx}\big(xC(x)\big)$, where $C(x)=\frac{1-\sqrt{1-4x}}{2x}$ is the ordinary generating function of the Catalan numbers, we can show that $T(r,1) = \binom{2r-1}{r}$.  This gives the identity 
$$\binom{2r-1}{ r}=\sum_{j=1}^r (r+1-j) C_{j-1}C_{r-j}\qquad \hbox{for } r\geq1.$$
\end{remark}

\begin{table}
$$\begin{array}{clllllllll}
\hline
r\setminus i & 0& 1& 2& 3& 4& 5& 6& 7& 8\\ \hline
0&	1 & & & & & & & & \\ 
1&	1 & 1 & & & & & & & \\ 
2&	2 & 3 & 3 & & & & & & \\ 
3&	5 & 10 & 16 & 16 & & & & & \\ 
4&	14 & 35 & 75 & 125 & 125 & & & & \\ 
5&	42 & 126 & 336 & 756 & 1296 & 1296 &  &  & \\ 
6&	132 & 462 & 1470 & 4116 & 9604 & 16807 & 16807 &  & \\ 
7&	429 & 1716 & 6336 & 21120 & 61440 & 147456 & 262144 & 262144 
    & \\ 
8&	1430 & 6435 & 27027 & 104247 & 360855 & 1082565 & 2657205 
    & 4782969 & 4782969 \\ \hline
\end{array}
$$
\caption{Values of the parking triangle $T(r,i)$ for $0\leq r,i\leq 8$.} 
\label{table:parkingtriangle}
\end{table}

%%%%%%%%%%%%%%%%%%%%%%%%%%%%%%
\subsection{Multi-labeled Dyck paths}\label{section:multilabeled_dyck_paths}
For $r\geq0$ and $0\leq i\leq r$, let $\M(r,i)$ be the set of labeled Dyck paths from $(0,0)$ to $(r,r)$ whose north steps are labeled by the multiset $\{0^{r-i}, 1,\ldots ,i\}$ such that the labels on consecutive north steps are non-decreasing.  We call these objects {\em multi-labeled Dyck paths}. It turns out that this family of objects is in bijection with the family of level-$i$ refined unified diagrams, and therefore provide a second combinatorial interpretation to the parking triangle numbers $T(r,i)$.

\begin{theorem} \label{theorem:bijectiondiagrams}
For $r\geq 0$, there is a bijection
\[
\Phi: \U_{\Car_{r+3}}^i(1,1,\ldots,1)\longrightarrow  \M(r,i).
\]
Hence, $|\M(r,i)|=T(r,i)$ for $0\leq i\leq r$. 
\end{theorem}

An informal description of our bijection is that the gravity diagram embedded below the lattice path tells where to add the north steps labeled by zeroes.  Every gravity diagram is simply a stack of left-adjusted lines---the north steps are added in columns corresponding to the other endpoints of those lines. We now share the complete proof details.

\begin{proof}To define this bijection we consider a larger family of diagrams that contains both a copy of $\U_{\Car_{r+3}}^i(1,1,\ldots,1)$ and an isomorphic copy of $\M(r,i)$. Let $i,j\ge 0$, $i+j\le r$. We define $\calH(r,i,j)$ to be the family of triples $U=(\vs,\sigma,D)$ satisfying the following properties.  The weak composition $\vs$ of   
$i+j$ represents the north steps of the lattice path of the form \[EN^{s_1}EN^{s_2}E\cdots N^{s_{r}}EE.\] 
The gravity diagram $D$ has lines extending from the first column of cells to the multiset of columns $\{1^{c_1},\ldots,r^{c_r}\}$ such that $\sum_{j=1}^r c_j =r-(i+j)$, where $c_1$ counts the isolated dots in the first column. The permutation $\sigma$ is a permutation of the multiset $\{0^j,1,2,\dots,i\}$ that is weakly increasing along north steps and with the additional condition that any occurrence of a $0$ in $\sigma$ corresponds to a north step immediately to the right of a column that is smaller than or equal to the number of dots in the shortest line in the gravity diagram $D$. The auxiliary diagrams in Figure \ref{figure:example_bijection_step_1} belong to $\calH(8,4,j)$ for $j=0,1,2,3$.

By definition, we have that $\calH(r,i,0)=\U_{\Car_{r+3}}^{i}(1,1,\ldots,1)$. When $i+j=r$ we have that the first east step of the Dyck path in a refined unified diagram occurs at level $r$, as in the proof of Proposition \ref{proposition:parkingtriangleparkingend}, and the associated gravity diagrams consist only of isolated dots. Therefore, each refined unified diagram is completely determined by its multi-labeled Dyck path, which begins after the first east step and ignores the last east step. Hence $\calH(r,i,r-i)$ is in bijection with $\M(r,i)$.

Now, for $j=1,\dots,r-i$, let $\Phi_j: \calH(r,i,j-1) \longrightarrow \calH(r,i,j)$ be the map defined for $(\vs,\sigma,D) \in  \calH(r,i,j-1)$ as follows.  We first let  $k=\min \{\ell\mid c_{\ell}\neq 0\}$ be the size of the shortest line in $D$.  We define $\Phi_j(\vs,\sigma,D)=(\vs',\sigma',D')$ by removing the shortest line in $D$ (which may be a dot) to obtain $D'$, 
and inserting an additional north step with a label $0$ immediately after the east step atop column $c_k$ to obtain $\vs'$ and $\sigma'$. 
See Figure \ref{figure:example_bijection_step_1} for examples of these maps.

We can readily verify that each $\Phi_j$ is well-defined and is a bijection; indeed the fact that there is a line of the gravity diagram below the Dyck path implies that when removed, we can decrease the $y$-coordinate of all the steps of the Dyck path lying above without crossing any other line of $D$, so that the new path remains a Dyck path. By choosing $k$ to be minimal we know that $\Phi_j(\vs,\sigma,D)\in \calH(r,i,j)$ and $\Phi_j$ is well defined. This is because the newly added $0$ appears immediately to the right of a column number that is smaller or equal to the number of dots in the shortest line in the gravity diagram $D'$. The inverse of $\Phi_j$ is easily defined as well by removing the rightmost 
$0$ at the bottom of a column from $\sigma'$ 
and its corresponding north step from $\vs'$, increasing the $y$-coordinate of all the steps in the Dyck path 
that occur before the removed north step, and adding a line with dots between columns $c_1$ and $c_k$, where $k$ is the column to the left of the north step removed. 

We then have that the map 
$$\Phi=\Phi_{r-i}\circ \cdots \circ \Phi_2 \circ \Phi_1$$ gives the desired bijection.  See Figure \ref{figure:example_bijection_step_1} for an example of this procedure.
\end{proof}

\begin{example}\label{ex:example_bijection}
Figure \ref{figure:example_bijection_step_1} depicts a level-$4$ refined unified diagram $U\in \U_{\Car_{11}}^4(1,1,\ldots,1)$ and its corresponding multi-labeled Dyck path $M\in\M(8,4)$ under the bijection $\Phi$.  Note that the four lines of the gravity diagram in $U$ extend into cells in the multiset of columns $\{2,2,3,6\}$, and these correspond to the locations of the zero labels in $M$.

\begin{figure}
    \centering
    \input{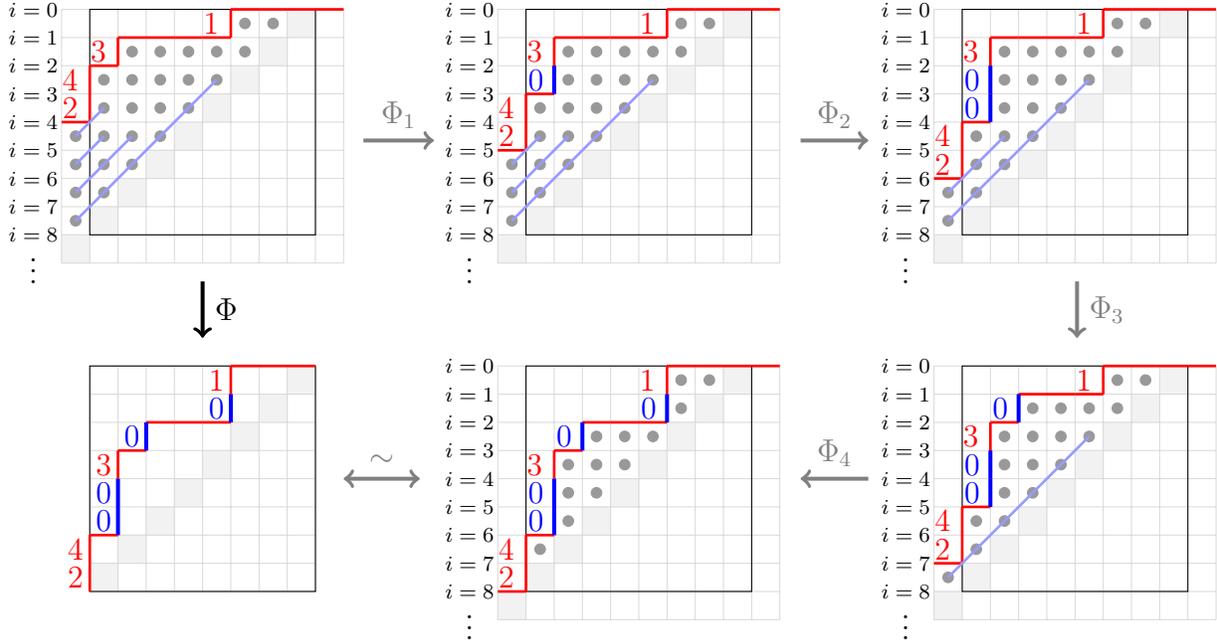}
    \caption{A level-$4$ refined unified diagram $U\in \U_{\Car_{11}}^4(1,1,\ldots,1)$, depicted in the upper lefthand corner, and its corresponding multi-labeled Dyck path $M\in\M(8,4)$, depicted in the bottom lefthand corner, via the bijection $\Phi=\Phi_4\circ\Phi_3\circ\Phi_2\circ\Phi_1$ in Theorem~\ref{theorem:bijectiondiagrams}.
    }
    \label{figure:example_bijection_step_1}
\end{figure}
\end{example}

\subsection{Counting multi-labeled Dyck paths}\label{section:counting_multilabeled_dyck_paths}

We now present a vehicle-parking scenario that models multi-labeled Dyck paths, analogous to the one for classical parking functions described in Section \ref{section:parkingfunctions}.  With this model, we are able to prove a closed-form formula for the entries $T(r,i)$ of the parking triangle.

Suppose that there are now $r$ parking spaces on a one-way street,
$r-i$ identical motorcycles $\textup{\small \faMotorcycle\,}_0,\ldots, \textup{\small \faMotorcycle\,}_0$ 
and $i$ distinct cars labeled $\textup{\small \faCar\,}_1,\ldots, \textup{\small \faCar\,}_i$. 
The vehicles have preferred parking spaces, and this information is recorded now as a {\em preference pair}, which contains a multiset of cardinality $r-i$, indicating parking preferences for the motorcycles, and a vector of length $i$ whose $k$-th entry contains the parking preference of the car $\textup{\small \faCar\,}_k$.  The vehicles advance down the street with the motorcycles parking first and the cars following in numerical order.  As in the classical case, all $r$ vehicles will find a parking space if and only if the preference pair can be uniquely represented by a multi-labeled Dyck path.

\begin{example} 
\label{ex:Motorcycles}
Let $r=8$ and $i=4$ so that the north labels are $\{0,0,0,0, 1,2,3,4\}$.  If the parking preference pair for four unordered motorcycles and four ordered cars is $\{2,2,3,6\}\times (6,1,2,1)$, then Figure~\ref{figure:parking_model_multilabels} illustrates
the multi-labeled Dyck path in $\M(8,4)$, which represents the parking preferences of the vehicles. 
The associated arrangement of parked vehicles is also shown.
\end{example}
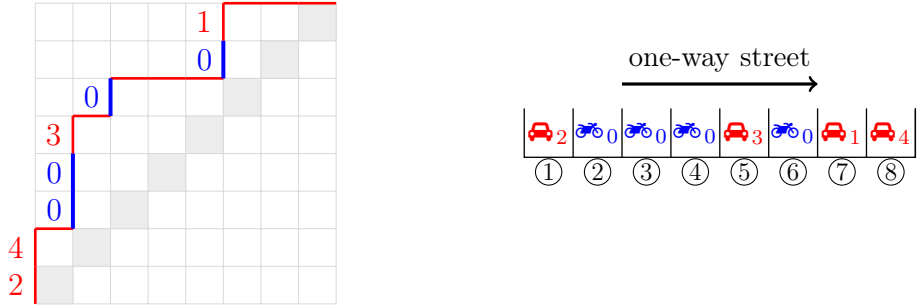
\begin{figure}[htb]
    \centering
    \begin{tikzpicture}
\begin{scope}[scale=0.5]
%define  the parameters
\edef \n{8}% This is the number of columns in the board
\edef \t{{1,1,1,1,1,1,1,1}}% This is the vector \t that gives the board
\edef \isdyck{1}%0 if you don't want a Dyck path 1 otherwise
\edef \pa{{2,3,1,0,0,2,0,0}}%  This is the Dyck path in the board
\edef \isparking{1}%0 if you don't want a parking function 1 otherwise
\edef \pf{{2,4,,,3,,,1}} % This is the parking function decoration
\edef \isgravity{0}%0 if you don't want the gravity diagram 1 otherwise, however you have to  manually edit the gravity diagram at the end of this group
\edef \isdiagonal{0}%0 if you don't want the diagonal 1 otherwise

%----------------------------------------------------------------------------
% Here comes the construction of the board
\pgfmathparse{\n-1}
\global\let\nminus\pgfmathresult
\edef\col{0}
\edef\row{0}
\foreach \x in {0,1,...,\nminus}{
\pgfmathparse{\t[\x]}
\global\let\part\pgfmathresult
\pgfmathsetmacro \newrow {\row+\part}
\pgfmathsetmacro \newcol {\col+1}
\draw[fill, color=gray!15,] (\col,\row) rectangle (\newcol,\newrow);
\global\let\col\newcol
\global\let\row\newrow
}
\draw[very thin, color=gray!30] (0, 0) grid (\col, \row);

%----------------------------------------------------------------------------
%draw path and parking decoration

\if \isdyck0
\else
\pgfmathsetmacro \colpa{0}
\pgfmathsetmacro \rowpa{0}
\pgfmathsetmacro \newrowpa{0}
\pgfmathsetmacro \newcolpa{0}
\foreach \x in {0,1,...,\nminus}{
\pgfmathsetmacro \part {\pa[\x]}
\if \part0
\else
\foreach \i in {1,...,\part}{
\pgfmathsetmacro \newrowpa {\rowpa+1}
\pgfmathsetmacro \lbpart {\pf[\rowpa]}
\if \isparking0]
\draw [line width=1, color=red] (\colpa, \rowpa)--(\newcolpa,\newrowpa);
\else
\draw [line width=1, color=red] (\colpa, \rowpa)--(\newcolpa,\newrowpa)node [midway, left] {\lbpart};
\fi
\global\let\rowpa\newrowpa
\global\let\newrowpa\rowpa
}
\fi
\pgfmathsetmacro \newcolpa{\colpa+1}
\draw [line width=1, color=red] (\colpa, \rowpa)--(\newcolpa,\newrowpa);
\global\let\colpa\newcolpa
\global\let\newcolpa\colpa
}
\fi

%----------------------------------------------------------------------------
%draw diagonal
 \if \isdiagonal0
\else
\draw (0,0)--(\col,\row);
\fi

%----------------------------------------------------------------------------
%%Draw gravity digram
%% Write the coordinates of the nodes you want and lines between them
%\if \isgravity0
%\else
%\tikzstyle{every node}=[circle, draw, fill,inner sep=0pt, minimum width=4pt,scale=1]
%\node at (0.5,3.5){};
%\node at (0.5,4.5){};
%\node at (1.5,4.5){};
%\node at (2.5,5.5){};
%\draw[line width=1] (0.5,3.5) -- (2.5,5.5);
%\fi

\tikzstyle{every node}=[circle,color=blue,inner sep=0pt, minimum width=4pt,scale=1]
\draw [line width=1.5, color=blue] (5,6)--(5,7);
\node at (4.5,6.5){$0$};
\draw [line width=1.5, color=blue] (2,5)--(2,6);
\node at (1.5,5.5){$0$};
\draw [line width=1.5, color=blue] (1,2)--(1,3);
\node at (0.5,2.5){$0$};
\draw [line width=1.5, color=blue] (1,3)--(1,4);
\node at (0.5,3.5){$0$};

\end{scope}

%This is the parking arrangement
\begin{scope}[scale=0.65]
%\node[]  at (4,-0.5) {$\mathbf{p}=\{2,2,3,4\}\times(6,1,2,1)$};

\draw[->, very thick]  (12,4.5)--(16,4.5) node [midway, above] {\small one-way street};
\tikzstyle{every node}=[circle,  draw, inner sep=1pt, minimum width=4pt,scale=0.8]
\draw[line width=0.5] (10,4) -- (10,3)--(18,3)--(18,4);
\draw[line width=0.5] (11,3)--(11,4);
\draw[line width=0.5] (12,3)--(12,4);
\draw[line width=0.5] (13,3)--(13,4);
\draw[line width=0.5] (14,3)--(14,4);
\draw[line width=0.5] (15,3)--(15,4);
\draw[line width=0.5] (16,3)--(16,4);
\draw[line width=0.5] (17,3)--(17,4);
\draw[line width=0.5] (18,3)--(18,4);
\node[]  at (10.5,2.7) {$1$};
\node[]  at (11.5,2.7) {$2$};
\node[]  at (12.5,2.7) {$3$};
\node[]  at (13.5,2.7) {$4$};
\node[]  at (14.5,2.7) {$5$};
\node[]  at (15.5,2.7) {$6$};
\node[]  at (16.5,2.7) {$7$};
\node[]  at (17.5,2.7) {$8$};

\tikzstyle{every node}=[inner sep=1pt,color=red, minimum width=4pt,scale=1]
\node[]  at (10.5,3.5) {$\textup{\tiny\faCar\,}_2$};
\node[]  at (14.5,3.5) {$\textup{\tiny\faCar\,}_3$};
\node[]  at (16.5,3.5) {$\textup{\tiny\faCar\,}_1$};
\node[]  at (17.5,3.5) {$\textup{\tiny\faCar\,}_4$};

\tikzstyle{every node}=[inner sep=1pt,color=blue, minimum width=4pt,scale=1]
\node[]  at (11.5,3.5) {$\textup{\tiny\faMotorcycle\,}_0$};
\node[]  at (12.5,3.5) {$\textup{\tiny\faMotorcycle\,}_0$};
\node[]  at (13.5,3.5) {$\textup{\tiny\faMotorcycle\,}_0$};
\node[]  at (15.5,3.5) {$\textup{\tiny\faMotorcycle\,}_0$};

\end{scope}
\end{tikzpicture}
    \caption{The multi-labeled Dyck path $\left((2,3,1,0,0,2,0,0),24003001\right)\in \M(8,4)$ corresponds to the parking preference pair $\mathbf{p}=\{2,2,3,6\}\times (6,1,2,1)$; the final parking arrangement is given on the right.  }
    \label{figure:parking_model_multilabels}
\end{figure}

\begin{theorem}\label{thm:triangleclosedformulas}
For all $r \geq 0$ and $0\leq i\leq r$,
\begin{equation}\label{equation:triangleclosedformulas}
T(r,i) = (r+1)^{i-1} \binom{2r-i}{r}.
\end{equation}
\end{theorem}

\begin{proof} 
We adapt an idea of Pollack~\cite[p.13]{FR74} to count parking preference pairs. 
Suppose that there are $r+1$ parking spaces on a circular one-way street whose entrance is before the first parking space, and that $r-i$ identical motorcycles and $i$ distinct cars arrive with their own parking preferences (including possibly the $(r+1)$-st space).
Since the street is circular, every preference pair will allow every vehicle to park, with one empty space left.
The circular preference pair is uniquely represented by a multi-labeled Dyck path if and only if the $(r+1)$-st space is the empty space.  

There is an action of the cyclic group $\mathbb{Z}_{r+1}$ on the set of parking preference pairs. Given $a\in \mathbb{Z}_{r+1}$ and the preference pair $\vp=\{p_1,\ldots, p_{r-i} \} \times (q_1,\ldots, q_i)$ in which the $j$-th vehicle ends up in space $S_j$, define  \[a\cdot \vp := \{p_1+a,\ldots, p_{r-i}+a\} \times (q_1+a,\ldots, q_i+a) \mod r+1.\] This is a parking preference pair in which the $j$-th vehicle is in space $(S_j+a) \mod r+1$, thus each orbit of this group action has $r+1$ elements, and the unique element with the $(r+1)$-st space empty corresponds to a multi-labeled Dyck path.

There are 
$\bbinom{r+1}{r-i} 
= \binom{(r+1)+(r-i)-1}{r-i} 
=\binom{2r-i}{r}$ multisets of preferences for the $r-i$ motorcycles, and $(r+1)^i$ preference vectors for the $i$ cars, and since each $\mathbb{Z}_{r+1}$-orbit has $r+1$ elements, there are $(r+1)^{i-1} \binom{2r-i}{r}$ multi-labeled Dyck paths in $\M(r,i)$.
\end{proof}

This was the last piece of the puzzle needed to prove the main theorem of this section.

\begin{theorem}\label{thm.main}
For $n\geq2$, the volume of the flow polytope $\F_{\Car_{n+1}}(1,1,\ldots,1)$ is
\[
\vol \F_{\Car_{n+1}}(1,1,\ldots,1) = C_{n-2}\cdot n^{n-2}.
\]
\end{theorem}

\begin{proof}
Theorem~\ref{theorem:unifieddiagrams}, Corollary~\ref{cor:decompositionofunifieddiagrams}, and Theorem~\ref{thm:triangleclosedformulas} combine to give
\begin{align*}
\vol \F_{\Car_{n+1}}(1,1,\ldots,1)
&= \big\lvert\U_{\Car_{n+1}}(1,1,\ldots,1)\big\rvert\\
&= \sum_{i=0}^{n-2} \binom{2n-4}{i} T(n-2,i)\\
&= \sum_{i=0}^{n-2} \binom{2n-4}{i} \binom{2n-4-i}{n-2}(n-1)^{i-1} \\
&=C_{n-2}\sum_{i=0}^{n-2} \binom{n-2}{i} (n-1)^i\\
&=C_{n-2}\cdot n^{n-2},
\end{align*}
where in the last step we have applied the binomial theorem.
\end{proof}

%%%%%%%%%%%%%%%%%%%%%%%%%%%%%%%%%%%%%%%%%%%%%%%%%%%%%%%%%%%%%
%%%%%%%%%%%%%%%%%%%%%%%%%%%%%%%%%%%%%%%%%%%%%%%%%%%%%%%%%%%%%
\section{Volumes of flow polytopes for other values of \texorpdfstring{$\va$}{a}}
\label{sec:othervols}

Our theory of refined unified diagrams allows us to calculate the volume of flow polytopes for a new family of net flow vectors of the form
\[\va=a\ve_1+b\vc,\]
for integers $a$ and $b$ and where $\vc$ is a zero-one vector with first entry equal to zero. Two fruitful cases are $\va=(a,0,\hdots,0,b,0,\hdots,0)=a\ve_1+b\ve_k$ for $2\leq k\leq n$ and $\va=(a,b,b,\hdots,b)$. 

Net flow vectors of this form are special because their refined unified diagrams satisfy the following enumerative property. 

\begin{lemma}\label{lem:UGic}
Let $G$ be a directed graph and let $\va=a\ve_1+b\vc$ be a net flow vector where $a$ and $b$ are non-negative integers and $\vc$ is a zero-one vector with first entry equal to zero. Then
\[\U_G^i(a\ve_1+b\vc)=b^i\,\U_G^i(\vc).\]
\end{lemma}
\begin{proof}
Because there are no north steps in the first column of a refined unified diagram, the initial entry of the net flow vector is irrelevant.  Therefore $\U_G^i(a\ve_1+b\vc)=\U_G^i(b\vc)$.  Since every entry of $b\vc$ is zero or $b$, then there are $b$ independent choices for net flow labels on each of the $i$ north steps of every refined unified diagram, so $\U_G^i(b\vc)=b^i\U_G^i(\vc)$.
\end{proof}

Combining Theorem~\ref{theorem:unifieddiagrams}, Proposition~\ref{p:iUnified}, and Lemma~\ref{lem:UGic} proves the following.

\begin{proposition}\label{prop:ab0b}
Let $G$ be a directed graph with $n+1$ vertices, $m$ edges, and shifted out-degree vector $\vt=(t_1,\ldots,t_n)$, and let $\va$ be a net flow vector of the form $\va=a\ve_1+b\vc$ for positive integers $a$ and $b$ and a zero-one vector $\vc$ with first entry equal to zero.  Then the volume of \(\F_G(\va)\) is 
\begin{equation} 
\vol\F_G(a\ve_1+b\vc)
=\sum_{i=0}^{m-n-t_1} \binom{m-n}{i}\cdot a^{m-n-i}b^i\cdot \big\lvert\U_G^i(\vc)\big\rvert.
\end{equation}
\end{proposition}

%%%%%%%%%%%%%%%%%%%%%%%%%%%%%%
\subsection{The case \texorpdfstring{$\va=(a,0,\hdots,0,b,0,\hdots, 0)$}{a=(a,0,...,0,b,0,...,0)}}

When $\va=a\ve_1+b\ve_k=(a,0,\hdots,0,b,0,\hdots,0)$, the quantity $\big\lvert\U_G^i(\ve_k)\big\rvert$ reduces to a {\em single} Kostant partition function, which makes finding an explicit formula for \(\F_G(\va)\) easier.  

\begin{proposition}\label{prop:a0b0} 
Let $G$ be a directed graph with $n+1$ vertices, $m$ edges, and shifted out-degree vector $\vt=(t_1,\ldots,t_n)$, and let $\va$ be a net flow vector of the form $\va=a\ve_1+b\ve_k$ for positive integers $a$ and $b$ and $2\leq k\leq n$.  Then the volume of \(\F_G(\va)\) is the binomial transform of a sequence of Kostant partition functions.  More precisely,
\begin{equation} 
\begin{aligned}
\vol&\,\F_G(a\ve_1+b\ve_k)=\\
&\sum_{i=0}^{m-n-t_1} \binom{m-n}{i}\cdot a^{m-n-i}b^i \cdot  K_{G\mid_n}(m-n-i-t_1,-t_2,\ldots,-t_{k-1},i-t_k,-t_{k+1},\ldots,-t_n).
\end{aligned}
\label{eq:vola0b0}
\end{equation}
\end{proposition}
\begin{proof}
The only refined unified diagrams that contribute to $\big\lvert\U_G^i(\ve_k)\big\rvert$ have the Dyck path $N^{m-n-i}E^{k-1}N^iE^{n-k+1}$ 
and trivial permutation $\sigma$. We conclude that
\[\big\lvert\U_G^i(\ve_k)\big\rvert=K_{G\mid_n}\big((m-n-i,0,\hdots,0,i,0,\hdots,0)-\vt\big).\] 
Apply Proposition~\ref{prop:ab0b} to complete the proof.
\end{proof}

To simplify notation, we define the sequence of \emph{refined Kostant constants}
$\kappa_{0},\ldots,\kappa_{m-n-t_1}$ by
\begin{equation} \label{def:bicase1100}
\kappa_i := K_{G|_n}\left(m-n-i-t_1,i-t_2,-t_3,\ldots,-t_n\right).
\end{equation}

Proposition~\ref{prop:a0b0} shows that the volume of $\F_G(1,1,0,\ldots,0)$ is simply a binomial transform of refined Kostant constants, which leads to elegant results for our flow polytope examples.

\begin{corollary}\label{cor:1100}
Given a directed graph $G$ with $n+1$ vertices, $m$ edges, and shifted out-degree vector $\vt=(t_1,\ldots,t_n)$, the volume of \(\F_G(1,1,0,\ldots,0)\) is a binomial transform of the sequence $\kappa_{0},\ldots,\kappa_{m-n-t_1}$: 
\[
\vol\F_G(1,1,0,\ldots,0)=\sum_{i=0}^{m-n-t_1} \binom{m-n}{i} \kappa_i.
\]
\end{corollary}

For the Pitman--Stanley graph, the refined Kostant constants are all equal to one, which is a consequence of Proposition~\ref{prop:GravityPS}.

\begin{proposition} For $n\geq 2$, 
\[
\vol \F_{\PS_{n+1}}(1,1,0,\ldots,0) = 2^{n-1}-1.
\]
\end{proposition}
\begin{proof}
For the Pitman--Stanley graph $\PS_{n+1}$, $m-n=n-1$ and $t_1=1$.  Since $\kappa_i=1$ for all $i$, so Corollary~\ref{cor:1100} implies
\[
\vol \F_{\PS_{n+1}}(1,1,0,\ldots,0) = \sum_{i=0}^{n-2}\binom{n-1}{i}= 2^{n-1}-1.\qedhere
\]
\end{proof}

We conjecture the following formula for the volume of the flow polytope involving the Pitman--Stanley graph and net flow vector $\va=(1,0,1,0,\hdots,0)$, which should follow from similar techniques.

\begin{conjecture}
For $n\geq 2$,
\[
\vol \F_{\PS_{n+1}}(1,0,1,0,\ldots,0) =  2^{n-1}-n-2.
\]
\end{conjecture} 

For the caracol graph, the refined Kostant constants are {\em generalized Catalan numbers} 
\[C_{n,k} = \binom{n+k}{k} - \binom{n+k}{k-1} = \frac{n-k+1}{n+1}\binom{n+k}{n} \textup{ for $0\leq k \leq n$.}\]  These numbers appear in the Catalan triangle~\cite[\href{http://oeis.org/A009766}{A009766}]{OEIS} given in Table~\ref{table:generalizedCatalan} in the Appendix.  It is known that $C_{n,k}$ is the number of lattice paths from $(0,0)$ to $(n,k)$ which do not rise above the line $y=x$.

\begin{lemma}\label{lemma:kCar}
Let $n\geq 2$.  For the caracol graph $\Car_{n+1}$,
\[\kappa_i= C_{n-2, n-2-i} = \frac{i+1}{n-1}\binom{2n-4-i}{n-2}\qquad \textit{ for }\qquad  0\leq i\leq n-2,\]
which are the entries of the  $(n-2)$-th row of the Catalan triangle \cite[\href{https://oeis.org/A009766}{A009766}]{OEIS}.  
\end{lemma}

\begin{proof}
We prove this by constructing a bijection between the set of gravity diagrams  \[\GD_{(\Car_{n+1})|_n}(n-2-i,i-1,-1,\ldots,-1,0)\] and lattice paths from $(0,0)$ to $(n-2, n-2-i)$ which do not rise above the line $y=x$. We follow a similar argument to the proof of Proposition~\ref{prop:Dyck1}.

Given a gravity diagram $D\in \GD_{(\Car_{n+1})|_n}(n-2-i, i-1,-1,\ldots,-1,0)$, reflect it about the vertical axis and embed it into the integer lattice $\mathbb{Z}\times \mathbb{Z}$ with the bottom row of $D$ occupying the lattice points $(1,1)$ through $(n-2,1)$. (If $i=n-2$, there will be no dot in position $(n-2,1)$.)
Construct the corresponding lattice path $P$ by starting at $(n-2,n-2-i)$ and taking vertical steps down in the same manner as in Proposition~\ref{prop:Dyck1}, by following horizontal line segments in the gravity diagram $D$ to their left-most endpoints, and taking a vertical step down to the row below. This process continues until we arrive at the left-most endpoint in the first row of $D$, and is completed by taking one final step downward and continuing the path leftward to $(0,0)$. Thus this set of gravity diagrams is in bijection with the lattice paths from $(0,0)$ to $(n-2,n-2-i)$ that do not rise above the line $y=x$, and the proposition follows.  
\end{proof}

\begin{remark}
The generalized Catalan numbers in Lemma~\ref{lemma:kCar} are also known as {\em ballot numbers}. This is due to a well-known bijection (for example, see Stanley~\cite[Exercise 2.168]{Stanley2015}), between lattice paths from $(0,0)$ to $(n,n)$ that do not rise above the line $y=x$, and the set of standard Young tableaux of shape $(n,n)$.  Through a generalization of this bijection, it can be seen that the refined Kostant constants $\kappa_i$ also count the set of standard Young tableaux of shape $(n-2,n-2-i)$.
\end{remark}

\begin{proposition} \label{prop:volCar1100}
For $n\geq 2$,
$$\vol \F_{\Car_{n+1}}(1,1,0,\ldots,0) = C_{n-2} \cdot n \cdot 2^{n-3}.$$
\end{proposition}
\begin{proof}
For the caracol graph $\Car_{n+1}$, $m-n=2n-4$ and $t_1=n-2$, so Corollary~\ref{cor:1100} and Lemma~\ref{lemma:kCar} combine to give
\begin{align*}
\vol \F_{\Car_{n+1}}(1,1,0,\ldots,0)
&= \sum_{i=0}^{n-2} \binom{2n-4}{i} \frac{i+1}{n-1} \binom{2n-4-i}{n-2} \\
&= C_{n-2} \sum_{i=0}^{n-2} \binom{n-2}{i} (i+1) = C_{n-2}\cdot n \cdot 2^{n-3}. 
\end{align*}
\end{proof}

For the zigzag graph, the refined Kostant constants are Entringer numbers $E_{n,k}$, introduced in Section~\ref{sec:gravityzigzag} and presented in Table~\ref{table:eulerbernoulli} in the Appendix.

\begin{lemma}
\label{lem:entringer}
For the zigzag graph $\Zig_{n+1}$, 
\[\kappa_i=E_{n-1,n-1-i}\qquad \textit{ for }\qquad 0\leq i\leq n-2,\]
which are the entries of the $(n-1)$-st row of the Euler--Bernoulli triangle. \cite[\href{https://oeis.org/A008282}{A008282}]{OEIS}.
\end{lemma}
\begin{proof} 
For the zigzag graph $\Zig_{n+1}$, $m-n=n-1$. The number \[\kappa_i = K_{\Zig_n}(n-2-i, i-1, -1,\ldots,-1,0)\] is the number of gravity diagrams of $\Zig_n$ with $n-2-i$ dots in the first column.  Thus it follows from Proposition~\ref{prop.zigzagrecurrence} and the recurrence equation for Entringer numbers that 
\[\kappa_i 
= \sum_{j=0}^{n-2-i}\left|\GD_{\Zig_n}(\vc',j)\right|
= \sum_{j=0}^{n-2-i} E_{n-2,n-2-j}
= E_{n-1,n-1-i}.\qedhere\]
\end{proof}

The binomial transform of the rows of the Euler--Bernoulli triangle are the Springer numbers. The $n$-th {\em Springer number}  $S_n$~\cite[\href{https://oeis.org/A001586}{A001586}]{OEIS} is the number of type $B_n$ snakes, which is an analogue of alternating permutations for signed permutations.  For $n\geq 1$, the first few values are $1,1,3,11,57,361,\ldots$.

\begin{theorem} \label{thm:volspringer}
For $n\geq 2$,
$$\vol \F_{\Zig_{n+1}}(1,1,0,\ldots,0) = S_n.$$
\end{theorem}
\begin{proof}
By the special case at $m=0$ for a result of Arnol'd~\cite[Theorem 21]{Arnold1992}, the binomial transform of a row of the Entringer numbers is a Springer number; more precisely,
\[\vol\F_{\Zig_{n+1}}(1,1,0,\ldots,0) = \sum_{i=0}^{n-2} \binom{n-1}{i} E_{n-1,n-1-i}=S_n.\qedhere\]
\end{proof}

The net flow vector $(1,1,0,\hdots,0)$ has previously been considered for the complete graph by Corteel, Kim, and M\'esz\'aros \cite{CKM}.  They used the Lidskii formula~\eqref{eq:lidskiivol} and constant term identities to derive the following product formula for the volume of $\mathcal{F}_{K_{n+1}}(1,1,0,\ldots,0)$.
It would be of interest to rederive this result using the refined Kostant constants.

\begin{theorem}[{\cite[Theorem 1.1]{CKM}}] \label{thm:volkn11}
Let $n\geq 1$. For the complete graph $K_{n+1}$,
\[
\vol \mathcal{F}_{K_{n+1}}(1,1,0,\ldots,0) = 2^{\binom{n}{2}-1} \cdot \prod_{i=1}^{n-2} C_i.
\]
\end{theorem}

%%%%%%%%%%%%%%%%%%%%%%%%%%%%%%
\subsection{The case \texorpdfstring{$\va=(a,b,\ldots,b)$}{a=(a,b,...,b)}}

We are also able to extract exact enumerative formulas from Proposition~\ref{prop:ab0b}  when $\va=a \ve_1+b(0,1,\hdots,1)=(a,b,\hdots,b)$.

We now recover a result first proved by Pitman and Stanley, which is a generalization of Equation~\eqref{equation:tree_number}.
For $n\geq1$ and $k\geq0$, let $A_{n,k}= (n-k)n^{k-1}$, the number of {\em acyclic functions} from $\{1,\ldots, k\}$ to $\{1,\ldots,n\}$.  These numbers appear in~\cite[\href{http://oeis.org/A058127}{A058127}]{OEIS}, and are presented in Table~\ref{table:acyclic} in the Appendix. These numbers satisfy the recurrence equation
$$
A_{n,k} = \sum_{j=0}^k\binom{k}{j}A_{n-1,j},
$$
with base cases $A_{n,0}=1$ for all $n\geq1$ and $A_{n,k}=0$ for $k\geq n$.

\begin{lemma} \label{lem.acyclic}
Let $n\geq 2$.  For the Pitman--Stanley graph $\PS_{n+1}$,
$$\big\lvert\U_{\PS_{n+1}}^i(0,1,\ldots,1)\big\rvert = (n-1-i)(n-1)^{i-1} = A_{n-1,i}.$$
\end{lemma}

\begin{proof}
Each refined unified diagram  $U=(\vs,\sigma,D) \in \U_{\PS_{n+1}}^i(a,1,\ldots,1)$ has only the trivial gravity diagram $D$ without lines,
since $\Phi_{(\PS_{n+1})|_n}^+ = \{\alpha_1,\ldots, \alpha_{n-1}\}$.
So $\U_{\PS_{n+1}}^i(0,1,\ldots,1)$ is the set of $(1^{n-1},0)$-parking functions whose Dyck path is of the form
$$\sigma =EN^{s_1}EN^{s_2}\cdots EN^{s_{n-2}}EN^0E,$$ 
which begins at $(0,n-i)$, and $\sum_{i=1}^{n-2} s_i = n-2$.

We decompose the set $\U_{\PS_{n+1}}^i(0,1,\ldots,1)$ according to the value of $s_1$.  If $s_1=i-j$, there are $\binom{i}{j}$ choices of parking labels on the north steps between the first and second east steps, and the path after the second east step is simply an element of $\U_{\PS_{n}}^i(0,1,\ldots,1)$.
Therefore,
$$\big\lvert\U_{\PS_{n+1}}^i(0,1,\ldots,1)\big\rvert = \sum_{j=0}^i \binom{i}{j} \big\lvert\U_{\PS_{n}}^j(0,1,\ldots,1)\big\rvert.$$
The base cases $\big\lvert\U_{\PS_{n+1}}^0\big\rvert=1$ for all $n$, and $\big\lvert\U_{\PS_{n+1}}^i\big\rvert=0$ for all $i\geq n-1$ are readily verified, since $\vt =(1^{n-1},0)$, completing the proof.
\end{proof}

Pitman and Stanley's result now follows by combining Proposition~\ref{prop:ab0b} and Lemma~\ref{lem.acyclic}.

\begin{proposition}\cite[Equation (7)]{PitmanStanley2002} \label{prop:abbcasePS}
For $n\geq2$ and positive integers $a$ and $b$,
$$\vol\F_{\PS_{n+1}}(a,b,\ldots,b) = a\big(a+(n-1)b\big)^{n-2}.$$
\end{proposition}

By applying Proposition~\ref{prop:ab0b} to the caracol graph, we obtain a generalization of Theorem~\ref{thm.main}.
\begin{theorem}\label{thm:Car-abb}
For $n\geq2$ and positive integers $a$ and $b$,
$$\vol\F_{\Car_{n+1}}(a,b,\ldots,b) 
= C_{n-2}\cdot  a^{n-2}\big(a+(n-1)b\big)^{n-2}.$$
\end{theorem}
\begin{proof} We compute 
\begin{align*}
\vol\F_{\Car_{n+1}}(a,b,\ldots,b)
&=\sum_{i=0}^{n-2} \binom{2n-4}{i} a^{2n-4-i} b^i \binom{2n-4-i}{n-2} (n-1)^{i-1}\\
&=C_{n-2} \cdot a^{n-2} \sum_{i=0}^{n-2} \binom{n-2}{i}
a^{n-2-i}b^i(n-1)^i\\
&=C_{n-2} \cdot a^{n-2} \big(a+(n-1)b\big)^{n-2}.\qedhere
\end{align*}
\end{proof}

The appearance of the sum of the entries of $\va$ in the quantity $(a+(n-1)b)^{n-2}$, in both Proposition~\ref{prop:abbcasePS} and Theorem~\ref{thm:Car-abb} is remarkable. However, from the data collected for $\Zig_{n+1}$, we know this is not a general phenomenon. Nonetheless, we give a conjectured formula for the volume of $\mathcal{F}_{\Car_{n+1}}(a,b,c,\ldots,c)$, verified up to $n=8$, that if true would imply Theorem~\ref{thm:Car-abb}.

\begin{conjecture}\label{conj:Car-abc}
For $n\geq2$ and positive integers $a$, $b$, and $c$,
$$\vol\F_{\Car_{n+1}}(a,b,c,\ldots,c) 
= C_{n-2} \cdot a^{n-2}(a+(n-1)b) (a+b+(n-2)c)^{n-3}.
$$
\end{conjecture}

%%%%%%%%%%%%%%%%%%%%%%%%%%%%%%%%%%%%%
%%%%%%%%%%%%%%%%%%%%%%%%%%%%%%%%%%%%%
\section{Geometric consequences}\label{sec:logconcave}

Recall that a sequence $s_0,s_1,\ldots,s_n$ of nonnegative integers is \emph{log-concave} if $s_i^2 \geq s_{i-1}s_{i+1}$ for $0<i<n$. In this section we show that sequences of refined Kostant constants  $\kappa_i$ and sequences $|\U_G^i(\vc)|$ (which appear in Section~\ref{sec:othervols}) are log-concave by identifying the corresponding flow polytopes as Minkowski sums of dilated polytopes and applying the {\em Aleksandrov--Fenchel inequalities} as in the work of Stanley~\cite{StAF} on order polytopes of posets.

Let $P$ and $Q$ be polytopes in $\mathbb{R}^k$. Given nonnegative real numbers $a$ and $b$, the {\em Minkowski sum} of the dilated polytopes  $aP$ and $bQ$ is the set
\[
aP + bQ = \{ ap+bq \mid p\in P, q \in Q\}.
\]

Baldoni and Vergne showed that flow polytopes with arbitrary net flow vectors can always be represented as Minkowski sums of dilated flow polytopes with elementary net flow vectors.

\begin{proposition}[{\cite[\S 3.4]{BV08}}] \label{prop:FPisMS}
Let $G$ be a directed graph on $n+1$ vertices and let $\va = (a_1,\ldots,a_n)$ be a nonnegative integer vector. The flow polytope  $\mathcal{F}_G(\va)$ can be written as the following Minkowski sum
\[
\mathcal{F}_G(\va) = a_1 \mathcal{F}_G(1,0,\ldots,0) + a_2 \mathcal{F}_G(0,1,0,\ldots,0) + \cdots + a_n \mathcal{F}_G(0,\ldots,0,1).
\]
\end{proposition}

The Aleksandrov--Fenchel inequalities proved independently by Alexandrov in \cite{Alexandrov} and Fenchel in \cite{Fe1,Fe2} relate the volume of a Minkowski sum of dilated polytopes to a sequence of log-concave constants $V_i$ known as {\em mixed volumes}.  

\begin{lemma}[Aleksandrov--Fenchel inequalities] \label{lem:AFIneq}
There are nonnegative constants $V_i$ for $i=0,\ldots,n$ such that
\[
\vol(aP+bQ) = \sum_{i=0}^k \binom{k}{i}\cdot a^{k-i}b^i\cdot V_i;
\]
furthermore, these constants are log-concave, i.e., for $i=1,\ldots,k-1$,
\[
V_i^2 \geq V_{i-1}V_{i+1}.
\]
\end{lemma}

For the net flow vectors of the form $\va=a\ve_1+b\vc$ that we investigated in Section~\ref{sec:othervols}, the log-concavity of these mixed volumes sequences allows us to prove that sequences $|\U_G^i(\vc)|$ (and consequently $\kappa_i$) are log-concave.  

\begin{theorem}\label{thm:logconcave}
Let $G$ be a directed graph on $n+1$ vertices and $m$ edges and let $\va$ be a net flow vector of the form $\va=a\ve_1+b\vc$ for integers $a$ and $b$ and a zero-one vector $\vc$ with first entry equal to zero.  Then the sequence of numbers $\big\lvert\U_G^i(\vc)\big\rvert$ for $0\leq i\leq m-n-t_1$ is log-concave.
\end{theorem}

\begin{proof}
By Proposition~\ref{prop:a0b0} we have that 
\begin{equation} \label{eq:Lidskiivola0b0}
\vol \mathcal{F}_G(a\ve_1+b\vc) = \sum_{i=0}^{m-n-t_1} \binom{m-n}{i}\cdot a^{m-n-i}b^{i}\cdot\big\lvert\U_G^i(\vc)\big\rvert. 
\end{equation}
By Proposition~\ref{prop:FPisMS} the polytope $\mathcal{F}_G(a\ve_1+b\vc)$ is the Minkowski sum of flow polytopes 
\[
\mathcal{F}_G(a\ve_1+b\vc) = a\mathcal{F}_G(\ve_1) + b\mathcal{F}_G(\vc).
\]
We view these flow polytopes as living in $\mathbb{R}^{m-n}$, the dimension of $\mathcal{F}_G(a\ve_1+b\vc)$. By Lemma~\ref{lem:AFIneq} we have that 
\begin{equation} \label{eq:mixedvola0b0}
\vol \mathcal{F}_G(a\ve_1+b\vc) = \sum_{i=0}^{m-n-t_1} \binom{m-n}{i} \cdot a^{m-n-i}b^{i}\cdot V_i.
\end{equation}
Moreover, the sequence $V_0,V_1,\ldots,V_n$ is log-concave. 

Next, we equate the RHS of Equations~\eqref{eq:Lidskiivola0b0} and \eqref{eq:mixedvola0b0} viewed as a polynomial in $a,b$ and conclude the equality of the coefficients:
\[
V_i = \big\lvert\U_G^i(\vc)\big\rvert.
\]
The log-concavity of the sequence  $\big\lvert\U_G^i(\vc)\big\rvert$ for $0\leq i\leq m-n-t_1$ then follows.
\end{proof}

When we specialize the result of Theorem~\ref{thm:logconcave} to $\va=a\ve_1+b\ve_k$, we learn that sequences of Kostant partition functions are log-concave.  And in particular, we see that the sequence of refined Kostant constants defined in Equation~\eqref{def:bicase1100} are log-concave.

\begin{corollary}\label{cor:logconcave.a0b0}
Let $G$ be a directed graph with $n+1$ vertices, $m$ edges, and shifted out-degree vector $\vt=(t_1,\hdots,t_n)$.  The sequence
\[
K_{G\mid_n}(m-n-i-t_1,-t_2,\ldots,-t_{k-1},i-t_k,-t_{k+1},\ldots,-t_n)
\]
for $0\leq i\leq m-n-t_1-\cdots-t_{k-1}$ is log-concave.  In particular, the sequence of refined Kostant constants $\kappa_i$ is log-concave.
\end{corollary}

We give two applications of this result that can also be obtained from the result of Stanley~\cite{StAF} for the $[2] \times [n-2]$ poset and the zigzag poset respectively.
First, applying Corollary~\ref{cor:logconcave.a0b0} to Lemma~\ref{lemma:kCar}, we recover the likely-known fact that the generalized Catalan numbers are log-concave. 

\begin{corollary}
Let $n$ be a nonnegative integer.  Then the sequence $C_{n,n}, C_{n,n-1},\ldots,C_{n,0}$ of generalized Catalan numbers is log-concave.
\end{corollary}

Applying Corollary~\ref{cor:logconcave.a0b0} to Lemma~\ref{lem:entringer}, we have the result that the Entringer numbers are log-concave.

\begin{corollary}
Let $n$ be a nonnegative integer.  Then the sequence $E_{n,n}, E_{n,n-1},\ldots,E_{n,0}$ of Entringer numbers is log-concave.
\end{corollary}

For the net flow vector $(a,b,\ldots,b)$, we recover the following corollary of Theorem~\ref{thm:logconcave}.

\begin{corollary}
 \label{cor:logconcaveabbbbb}
 For any directed graph $G$ with $n+1$ vertices, $m$ edges, and shifted out-degree vector $\vt=(t_1,\ldots,t_n)$, the sequence  $\big\lvert\U_G^i(0,1,\hdots,1)\big\rvert$  for $0\leq i\leq m-n-t_1$ is log-concave.  In particular, the sequence $T(n,n), T(n,n-1),\ldots,T(n,0)$ defined in Equation~\eqref{equation:triangleclosedformulas} is log-concave.
\end{corollary}

If we apply Corollary~\ref{cor:logconcaveabbbbb} to Theorem~\ref{theorem:bijectiondiagrams}, we learn about the log-concavity of the rows of the parking triangle.  (We recall that $\big\lvert\U_G^i(0,1,\hdots,1)\big\rvert=\big\lvert\U_G^i(1,1,\hdots,1)\big\rvert$ since the first net flow entry $a_1$ is irrelevant.)

\begin{corollary} For $n\geq 2$, the sequence $T(n,n), T(n,n-1),\ldots,T(n,0)$ defined in Equation~\eqref{equation:triangleclosedformulas} is log-concave.
\end{corollary}

And if we apply Corollary~\ref{cor:logconcaveabbbbb} to Lemma~\ref{lem.acyclic}, we prove that the sequence $A_{n,k}$ of the number of acyclic functions is also log-concave.  

\begin{corollary}
The sequence $A_{n,n}, A_{n,n-1},\ldots,A_{n-2,0}$ is log-concave.
\end{corollary}

\section{A new polynomial for the volume of flow polytopes}
\label{sec:polynomial}

In this section we describe a volume function of flow polytopes with similar properties to the Ehrhart polynomial of a polytope, including polynomiality. We also give a conjecture for this function that would give a new proof of Theorem~\ref{thm.main}.

Given a directed graph $G$ with $n+1$ vertices labeled by $[n+1]$ and a nonnegative integer $x$, we denote by $\widehat{G}(x)$ the directed graph obtained by adding a vertex $0$ and $x$ edges $(0,i)$ for $i=1,\ldots,n$.  We define $E_G(x)$ to be the volume of the flow polytope $\mathcal{F}_{\widehat{G}(x)}(1,0,\ldots,0)$; that is, 
\begin{equation}
    E_G(x):=\vol \mathcal{F}_{\widehat{G}(x)}(1,0,\ldots,0).
\end{equation}
Furthermore, for a directed graph $G$ with vertices $[n+1]$, we will also define its {\em shifted in-degree vector } $\vu=(u_1,\ldots, u_n)$ to be the vector whose $i$-th entry is one less than the in-degree of vertex $i$.

\begin{theorem} \label{thm:volpoly}
Let $G$ be a directed graph on $n+1$ vertices and $m$ edges with shifted in-degree vector $\vu=(u_1,\ldots, u_n)$. 
The function $E_G(x)$ is the Kostant partition function 
\begin{equation} \label{eq:thmvolpolyK}
E_G(x)  = K_{G}\Bigl(x+{u}_1, x+{u}_2,\ldots,x+{u}_n,-xn-\sum_{k=1}^n {u}_k\Bigr),
\end{equation}
which is a polynomial in $x$ of degree $m-n$ with nonnegative coefficients and whose leading coefficient equals
\begin{equation} \label{eq:thmleadcoeff}
\frac{1}{(m-n)!}\vol\mathcal{F}_G(1,1,\ldots,1).
\end{equation}
\end{theorem} 

\begin{remark}
There are two reasons that we feel that the polynomial $E_G(x)$ can be viewed as related to an Ehrhart polynomial.  First, if we remove the shifted in-degree terms, the quantity $K_G(x,\hdots,x,-xn)=xK_G(1,\hdots,1,-n)$ is the Ehrhart polynomial of the flow polytope $\mathcal{F}_G(1,1,\ldots,1)$.  In addition, Equation~\eqref{eq:thmleadcoeff} shows that the leading coefficient of this polynomial is  the volume of a flow polytope divided by the factorial of the degree of the polynomial.  However, in contrast to an Ehrhart polynomial, the volume being calculated here is of a related polytope and not of the polytope itself. Note that the polynomial $E_G(x)$ has nonnegative coefficients; however, it is not known in general when the Ehrhart polynomial of a flow polytope has positive coefficients \cite{LiuSurvey}.
\end{remark}

The proof of Theorem~\ref{thm:volpoly} requires a variant of Equation~\eqref{eq:lidskiivol} in terms of in-degrees and a Lidskii formula for the Kostant partition function $K_G(\cdot)$ from~\cite{BBDRRV14}. 

\begin{lemma}[{Postnikov--Stanley (unpublished), \cite[Theorem 6.1]{MM1}}] \label{lem:volindeg}
Let $G$ be a directed graph on $n+1$ vertices with shifted in-degree vector $\vu=(u_1,\ldots, u_n)$.  Then
\[
\vol\mathcal{F}_G(1,0,\ldots,0) = K_G(0,u_2,u_3,\ldots,u_{n}, -m+n+u_{n+1}).
\]
\end{lemma}

\begin{lemma}[Lidskii formula for $K_G$ {\cite[Theorem 38]{BV08}}, {\cite[Theorem 1.1]{MM2}}] \label{lem:lidskiiKost}
Let $G$ be a directed graph on $n+1$ vertices with shifted out-degree vector $\vt=(t_1,\ldots, t_n)$ and shifted in-degree vector $\vu=(u_1,\ldots, u_n)$, and let $\va = (a_1,\ldots,a_n)$ be a nonnegative integer vector.  Then
\[
K_G(\va) = \sum_{\vs\rhd \vt}
\bbinom{\va - \vu}{\vs} \cdot  K_{G|_n}(\vs - \vt), 
\]
where the sum is over weak compositions $\vs= (s_1,\ldots, s_n)$ of $m-n$  
that dominate $\vt$, and where $\bbinom{\va-\vu}{\vs} := \bbinom{a_1-u_1}{s_1} \cdots \bbinom{a_n-u_n}{s_n}$.
\end{lemma}

\begin{proof}[Proof of Theorem~\ref{thm:volpoly}]
Let $\widehat{G}(x)$ have shifted in-degree vector $\widehat{\vu}=(\widehat{u}_0,\widehat{u}_1,\ldots, \widehat{u}_n)$, so that $\widehat{u}_i = x + u_{i}$. By Lemma~\ref{lem:volindeg} the volume of the flow polytope $\mathcal{F}_{\widehat{G}(x)}(1,0,\hdots,0)$ is given by
the value of the Kostant partition function 
\[
K_{\widehat{G}(x)}\Big(0,\widehat{u}_1,\ldots,\widehat{u}_{n},-\sum_{i=0}^{n} \widehat{u}_i \Big).
\]
Since the net flow on the zeroth vertex is zero, this quantity simplifies to 
\[
K_G\Big(x+u_1,x+u_2,\ldots,x+u_n, -xn - \sum_{k=1}^n u_k\Big),
\]
which proves Equation~\eqref{eq:thmvolpolyK}. To show that $E_G(x)$ is a polynomial with the stated properties, we apply Lemma~\ref{lem:lidskiiKost} to Equation~\eqref{eq:thmvolpolyK} to find that
\[
E_G(x) =  \sum_{{\vs\rhd \vt}}
\bbinom{x}{s_1}\cdots
\bbinom{x}{s_{n}} \cdot  K_{G|_n}(\vs - \vt).
\]
Since each $\bbinom{x}{s_i}$ is a polynomial in $x$ with nonnegative coefficients and $K_{G|_n}(\vs - \vt) \geq 0$ then $E_G(x)$ is also a polynomial in $x$ with nonnegative coefficients. The degree of $E_G(x)$ is $m-n$ and its leading coefficient is
\[
\sum_{{\vs\rhd \vt}} \frac{1}{s_1!}\cdots \frac{1}{s_n!}
\cdot  K_{G|_n}({\bf s} - {\bf t}) = 
\frac{1}{(m-n)!}\sum_{{\vs\rhd \vt}} \binom{m-n}{\vs}
\cdot  K_{G|_n}({\bf s} - {\bf t})
\]
By the Lidskii formula of Equation~\eqref{eq:lidskiivol}, the sum on the right hand side is the volume of $\mathcal{F}_G(1,1,\ldots,1)$, which proves Equation~\eqref{eq:thmleadcoeff}.
\end{proof}

A result equivalent to Theorem~\ref{thm:volpoly} for the complete graph appeared in work of  Zhou--Lu--Fu~\cite[Lemma 2.1]{ZLF}. The result there is written in terms of constant term identities. An explicit product formula for $E_{K_{n+1}}(x)$ has appeared in~\cite{BV08} and~\cite{MeszarosProd} and was proved using the Morris constant term identity. This formula involves a product of Catalan numbers and Proctor's formula \cite{Pro} for the number of plane partitions of staircase shape $(n-1,n-2,\ldots,2,1)$ with entries at most $x-1$. 

\begin{theorem}[{\cite[Proposition 26]{BV08}, \cite[Equation (8)]{MeszarosProd}}]
\begin{equation} \label{eq:EGKn}
E_{K_{n+1}}(x) = \prod_{i=1}^{n-1} C_i \cdot \prod_{1\leq i<j\leq n} \frac{2(x-1) + i+j-1}{i+j-1}.
\end{equation}
\end{theorem}

\begin{corollary}[{\cite[Theorem 1.8]{MMR}, \cite[Lemma 2.1]{ZLF}}]
\[
\vol \mathcal{F}_{K_{n+1}}(1,1,\ldots,1) = {\prod_{i=1}^{n-1} C_i} \cdot \frac{\binom{n}{2}!}{\prod_{i=1}^{n-2} (2i + 1)^{n - i - 1}}.
\]
\end{corollary}

\begin{proof}
By Theorem~\ref{thm:volpoly} the volume of $\mathcal{F}_{K_{n+1}}(1,1,\ldots,1)$ equals the product of $\binom{n}{2}!$ with the leading coefficient of $E_{K_{n+1}}(x)$. The leading coefficient of Equation~\eqref{eq:EGKn} is \[ \prod_{i=1}^{n-1} C_i\cdot 2^{\binom{n}{2}} \cdot \prod_{i=1}^{n-1} \frac{i!}{(2i)!},\] from which the result follows.
\end{proof}

Next, we give a product formula for $E_{\PS_{n+1}}(x)$ that follows from known formulas for the number of lattice points of the Pitman--Stanley polytope. Note that in \cite[Corollary 16]{MeszarosProd} M\'esz\'aros shows that the {\em Fuss--Catalan numbers} give the volume of a very similar polytope to $\mathcal{F}_{\widehat{\PS_{n+1}}(x)}(1,0,\ldots,0)$. The difference is that in that setting one additional edge $(0,1)$ is added to the graph.

\begin{proposition}
\label{conj:prePS}
For the Pitman--Stanley graph $\PS_{n+1}$ we have that 
\begin{equation} \label{eq:volprePS}
E_{\PS_{n+1}}(x) = \frac{1}{xn+n-1} \binom{xn +n-1}{n}.
\end{equation}
\end{proposition} 

\begin{proof}
Let $x$ be a nonnegative integer. 
Equation~\eqref{eq:thmvolpolyK} gives $E_{\PS_{n+1}}(x)$ as the number of lattice points of a Pitman--Stanley polytope,  
\[
E_{\PS_{n+1}}(x) = K_{\PS_{n+1}}(x-1,x,\ldots,x,-xn+1).
\]
By its definition \cite{PitmanStanley2002}, this is the number of tuples $(r_1,\ldots,r_n)$ of nonnegative integers satisfying 
\begin{align*}
r_1 &\leq x-1,\\ 
r_1 + r_2 &\leq 2x-1,\\
& \vdots \\
r_1+\cdots + r_{n-1} &\leq (n-1)x-1,\\ 
r_1 + \cdots + r_{n-1} + r_n &=nx-1.
\end{align*}
Such tuples are in correspondence with weak compositions 
\[{\bf s} = (s_1,\ldots,s_{n-1}) \rhd (x,\ldots,x,x-1),\] which are in correspondence with rational $(n,xn-1)$-Dyck paths.   Rational $(a,b)$-Dyck paths are counted by the rational $(a,b)$-Catalan numbers, $\frac{1}{b}\binom{a+b}{b}$; see, for example, \cite{ALW}. 
\end{proof}

Here is a conjectured product formula for $E_{\Car_{n+1}}(x)$.

\begin{conjecture} \label{conj:preCaracol}
For the Caracol graph $\Car_{n+1}$ we have that 
\begin{equation} \label{eq:volpreCaracol}
E_{\Car_{n+1}}(x) = \frac{1}{xn+n-3}\binom{xn+2n-5}{n-1} \binom{x+n-3}{n-2}.
\end{equation}
\end{conjecture}

The above conjectures were found through numerical experiments in Maple by calculating $E_G(x)$ for many values of $x$.  Besides the intrinsic interest in an Ehrhart-like polynomial for flow polytopes, a proof of these conjectures would give new proofs of the formula~\eqref{equation:PitmanStanleyvolume} and of Theorem~\ref{thm.main}.

\begin{proposition}
Conjecture~\ref{conj:prePS} implies the volume formula~\eqref{equation:PitmanStanleyvolume} of Pitman and Stanley for $\mathcal{F}_{\PS_{n+1}}(1,1,\ldots,1)$.
\end{proposition}

\begin{proof}
The leading coefficient of the conjectured formula for $E_{\PS_{n+1}}(x)$ in Equation~\eqref{eq:volprePS} is $\frac{n^{n-2}}{(n-1)!}$, so by Theorem~\ref{thm:volpoly} we would have that $\vol \mathcal{F}_{\PS_{n+1}}(1,1,\ldots,1) = n^{n-2}$
as claimed.
\end{proof}

\begin{proposition}
Conjecture~\ref{conj:preCaracol} implies Theorem~\ref{thm.main}.
\end{proposition}

\begin{proof}
The leading coefficient of the conjectured formula for $E_{\Car_{n+1}}(x)$ in Equation~\eqref{eq:volpreCaracol} is $\frac{n^{n-2}}{(n-1)!(n-2)!}$, so by Theorem~\ref{thm:volpoly} we would have that
\[\vol \mathcal{F}_{\Car_{n+1}}(1,1,\ldots,1) = \big((3n-4)-n\big)!\frac{n^{n-2}}{(n-1)!(n-2)!}=C_{n-2}\cdot n^{n-2},
\]
which, in turn, proves Theorem~\ref{thm.main}.
\end{proof}

%%%%%%%%%%%%%%%%%%%%%%%%%%%%%%%%%%%%%%%%%%%%%%%%%%%%%%%%%%%%%
%%%%%%%%%%%%%%%%%%%%%%%%%%%%%%%%%%%%%%%%%%%%%%%%%%%%%%%%%%%%%
\section*{Acknowledgments}

This project was initiated at the Polyhedral Geometry and Partition Theory workshop at the American Institute of Mathematics in November 2016.  We are extremely grateful to the organizers of the workshop---Federico Ardila, Benjamin Braun, Peter Paule, and Carla Savage---and to the American Institute of Mathematics for their funding which led to this collaboration and future related opportunities. We also thank Sylvie Corteel and Karola M\'esz\'aros for stimulating discussions and Richard Stanley for help with references in Section~\ref{sec:logconcave}. C. Benedetti thanks the Faculty of Science of Universidad de los Andes, York University, and Fields Institute for their support. R.\ S.\ Gonz\'alez D'Le\'on was supported during this project by University of Kentucky, York University and Universidad Sergio Arboleda and he is grateful for their support.
C.\ R.\ H.\ Hanusa is grateful for the support of PSC-CUNY Award 69120-0047. P.\ E.\ Harris was supported by NSF award DMS-1620202. A.\ Khare is partially supported by Ramanujan Fellowship SB/S2/RJN-121/2017 and MATRICS grant MTR/2017/000295 from SERB (Government of India), by grant F.510/25/CAS-II/2018(SAP-I) from UGC (Government of India), and by a Young Investigator Award from the Infosys Foundation.
A.\ H.\ Morales was partially supported by an AMS-Simons travel grant. M.\ Yip was partially supported by a Simons collaboration grant 429920.  Finally, we appreciate the close reading of this work by the referees.

%%%%%%%%%%%%%%%%%%%%%%%%%%%%%%%%%%%%%%%%%%%%%%%%%%%%%%%%%%%%%
%%%%%%%%%%%%%%%%%%%%%%%%%%%%%%%%%%%%%%%%%%%%%%%%%%%%%%%%%%%%%
%%%%%%%%%%%%%%%%%%%%%%%%%%%%%%%%%%%%%%%%%%%%%%%%%%%%%%%%%%%%%
%%%%%%%%%%%%%%%%%%%%%%%%%%%%%%%%%%%%%%%%%%%%%%%%%%%%%%%%%%%%%
%%%%%%%%%%%%%%%%%%%%%%%%%%%%%%%%%%%%%%%%%%%%%%%%%%%%%%%%%%%%%
\appendix

\section*{Appendix. Triangles of Integers}
\begin{table}[ht]
$$\begin{array}{clllllllll}
\hline
E_{n,k} & 0& 1& 2& 3& 4& 5& 6& 7&\\ \hline
0&	0 & & & & & & & & \\ 
1&	0 & 1 & & & & & & & \\ 
2&	0 & 1 & 1 & & & & & & \\ 
3&	0 & 1 & 2 & 2 & & & & & \\ 
4&	0 & 2 & 4 & 5 & 5 & & & & \\ 
5&	0 & 5 & 10 & 14 & 16 & 16 &  &  & \\ 
6&	0 & 16 & 32 & 46 & 56 & 61 & 61 &  & \\ 
7&	0 & 61 & 122 & 178 & 224 & 256 & 272 & 272 
    & \\ \hline
\end{array}
$$
\caption{Entries of the Euler--Bernoulli triangle (the Entringer numbers $E_{n,k}$)
\cite[\href{https://oeis.org/A008282}{A008282}]{OEIS}.  They appear in Proposition~\ref{prop.zigzagrecurrence} and Lemma~\ref{lem:entringer}.}
\label{table:eulerbernoulli}
\end{table}

\begin{table}[ht]
$$\begin{array}{clllllllll}
\hline
C_{n,k} & 0& 1& 2& 3& 4& 5& 6& 7&\\ \hline
0&	1 & & & & & & & & \\ 
1&	1 & 1 & & & & & & & \\ 
2&	1 & 2 & 2 & & & & & & \\ 
3&	1 & 3 & 5 & 5 & & & & & \\ 
4&	1 & 4 & 9 & 14 & 14 & & & & \\ 
5&  1 & 5 & 14 & 28 & 42 & 42 &  &  & \\ 
6&	1 & 6 & 20 & 48 & 90 & 132 & 132 &  & \\ 
7&	0 & 7 & 27 & 75 & 165 & 297 & 429 & 429 
    & \\ \hline
\end{array}
$$
\caption{Entries of the Catalan triangle $C_{n,k}$
\cite[\href{https://oeis.org/A009766}{A009766}]{OEIS}.  They appear in Lemma~\ref{lemma:kCar}.}
\label{table:generalizedCatalan}
\end{table}

\begin{table}[ht]
$$\begin{array}{clllllllll}
\hline
A_{n,k} & 0& 1& 2& 3& 4& 5& 6& 7&\\ \hline
1&	1 & & & & & & & & \\ 
2&	1 & 1 & & & & & & & \\ 
3&	1 & 2 & 3 & & & & & & \\ 
4&	1 & 3 & 8 & 16 & & & & & \\ 
5&	1 & 4 & 15 & 50 & 125 & & & & \\ 
6&  1 & 5 & 24 & 108 & 432 & 1296 &  &  & \\ 
7&	1 & 6 & 35 & 196 & 1029 & 4802 & 16807 &  & \\ 
    \hline
\end{array}
$$
\caption{Entries of the acyclic function triangle $A_{n,k}$
\cite[\href{https://oeis.org/A058127}{A058127}]{OEIS}.  They appear in Lemma~\ref{lem.acyclic}.}
\label{table:acyclic}
\end{table}

\end{document}